\documentclass[a4paper,10pt]{article}

\usepackage{amssymb,amsmath,amsthm}
\usepackage[all,cmtip,line]{xy}
\usepackage[T1]{fontenc} % For Guillemots

\newtheorem{Theorem}{Theorem}[section]
\newtheorem{Lemma}[Theorem]{Lemma}
\newtheorem{Corollary}[Theorem]{Corollary}
\newtheorem{Proposition}[Theorem]{Proposition}

\theoremstyle{definition}
\newtheorem{Definition}[Theorem]{Definition}
\newtheorem{Example}[Theorem]{Example}
\newtheorem{Remark}[Theorem]{Remark}

\newcommand{\be}{\begin{equation}}
\newcommand{\ee}{\end{equation}}

\newcommand{\one}{\mathbf{1}}

\newcommand{\ol}[1]{\ensuremath{\overline{#1}}}
\newcommand{\oo}{\infty}

\newcommand{\ham}{\mathrm{ham}}
\newcommand{\smp}{\mathrm{sp}}

\newcommand{\N}{\ensuremath{\mathbb{N}}}
\newcommand{\Z}{\ensuremath{\mathbb{Z}}}

\newcommand{\R}{\ensuremath{\mathbb{R}}}

\newcommand{\vol}{\ensuremath{\omega^{n}/n!}}

\newcommand{\fg}{\mathfrak{g}}
\newcommand{\fh}{\mathfrak{h}}
\newcommand{\hfg}{\widehat{\mathfrak{g}}}

\newcommand{\n}{\mathfrak{n}}

\newcommand \al{\alpha}
\newcommand\ga{\gamma}
\newcommand\de{\delta}

\renewcommand\th{\theta}

\newcommand\la{\lambda}

\newcommand\si{\sigma}
\newcommand\ta{\tau}

\newcommand\ps{\psi}
\newcommand\om{\omega}

\newcommand\Th{\Theta}
\newcommand\La{\Lambda}

\newcommand\Om{\Omega}

\def\RR{\mathbb R}

\newcommand{\dd}{{\tt d}}

\newcommand\x{\times}
\newcommand\on{\operatorname}

\newcommand\dR{\on{dR}}

\newcommand\can{\on{can}}

\newcommand\g{\mathfrak g}

\newcommand\h{\mathfrak h}

\begin{document}

\title{Universal Central Extension of the Lie Algebra of Hamiltonian Vector Fields}
\author{Bas Janssens\footnote{Mathematisch Instituut, Universiteit Utrecht, 3584 CD Utrecht, the Netherlands. Email: {\tt B.Janssens@uu.nl}}, Cornelia Vizman\footnote{Department of Mathematics,
West University of Timi\c{s}oara,
300223-Timi\c{s}oara
Romania. Email: {\tt cornelia.vizman@e-uvt.ro}}}
\maketitle
\begin{abstract}
For a connected symplectic manifold $X$,
we determine the universal central extension of the Lie algebra $\ham(X)$ of hamiltonian vector fields.
We classify the central extensions of $\ham(X)$, of the Lie algebra $\smp(X)$ of symplectic vector fields, of
the Poisson Lie algebra $C^{\infty}(X)$, and of its compactly supported version $C^{\infty}_{c}(X)$.
\end{abstract}

\section{Introduction}

In this paper, we classify continuous central extensions 
of several infinite dimensional Lie algebras associated to a symplectic manifold $(X,\omega)$;
the Poisson Lie algebra $C^{\infty}(X)$,  the compactly supported Poisson Lie algebra $C^{\infty}_{c}(X)$, 
the Lie algebra of hamiltonian vector fields $\ham(X)$,
and the Lie algebra $\mathrm{sp}(X)$ of symplectic vector fields.

If $\fg$ is the Lie algebra of a locally convex Lie group $G$, then
central extensions of $\fg$ by $\R$ are
the infinitesimal versions of central extensions of $G$ by $U(1)$, whose
classification plays
a pivotal role in projective unitary representation theory 
%infinite dimensional Lie groups 
\cite{PressleySegal1986, Neeb02, JanssensNeeb2015}.
If $X$ is compact, then the Lie algebra of hamiltonian or sym\-plectic vector fields is the 
Lie algebra of the group of hamiltonian or sym\-plectic diffeomorphisms, respectively.
If, moreover, $\omega$ is integral, then the Poisson Lie algebra is the Lie algebra of the quantomorphism group
\cite{Schmid1978, RatiuSchmid1981}. 
Although the link with Lie groups is important for applications, 
the present paper is only concerned with Lie algebras,
so we will be able to handle noncompact as well as compact manifolds $X$.
The integrability issue is addressed in a forthcoming paper.

Continuous central extensions of a locally convex Lie algebra $\fg$ by $\R$ are classified by 
the continuous second Lie algebra cohomology $H^2(\fg, \R)$ with coefficients in $\R$. 
Therefore, a large part of this paper 
is devoted to determining this cohomology group for the above mentioned Lie algebras associated to a symplectic manifold $X$. 

This problem %The problem of determining $H^2(\fg, \R)$ 
is somewhat different in flavour from determining the cohomology with 
adjoint coefficients, cf.\ \cite{Lichnerowicz1973, Lichnerowicz1974H1DIFF, 
DeWildeLecomte1983, DeWildeLecomte1983H3DIFF, DeWildeLecomteGutt1984}, in the sense that 
our 2-cocycles come from distributions on $X \times X$ that can be
singular in all directions.  
Cohomology with real coefficients was studied in \cite{GelfandKalininFuks1972, GelfandMathieu1992}
for $\ham(X)$, where $X$ is a formal neighbourhood or a torus, respectively.
Lie algebra cohomology for various Lie algebras of vector fields is treated in \cite{Fu}.  
Our description of the second cohomology 
goes back to Kirillov \cite{kir} 
(for dimension 2)
and Roger \cite{Roger1995} (for arbitrary dimension), 
but no proof has appeared in the literature to the best of our knowledge.

\subsection{Summary of the Results}

We give a brief summary of our main results: the 
continuous second Lie algebra cohomology 
of 
$C^{\infty}_{c}(X)$, $C^{\infty}(X)$, $\ham(X)$, and $\mathrm{sp}(X)$, 
and the universal central extension of $\ham(X)$ for a symplectic 
manifold $(X,\om)$.
It turns out that for $X$ noncompact, the same Lie algebra serves as the universal central extension
of the Poisson Lie algebra $C^{\infty}(X)$. 
 
 \subsubsection{Continuous second Lie algebra cohomology}

The second Lie algebra cohomology of $C^{\infty}_{c}(X)$, $C^{\infty}(X)$, and $\ham(X)$
is determined in Section \ref{s4}.
The case $\mathrm{sp}(X)$ is based on this, but requires also some different 
methods. It is done in Section \ref{s6}.

\paragraph{Compactly supported Poisson Lie algebra.}
Let $C^{\infty}_{c}(X)$ be the Lie algebra of 
compactly supported, smooth functions on a symplectic manifold,
equipped with the Poisson bracket.
Then Theorem~\ref{HoofdZaak} shows that
the central extensions of $C^{\infty}_{c}(X)$ by $\R$ are classified by
the first de Rham cohomology of $X$,
	\be
	H^2(C_{c}^{\infty}(X), \R) \simeq H^1_{\mathrm{dR}}(X)\,.
	\ee
This theorem underlies all other results in this paper.
%The isomorphism is quite explicit, Thm.~\ref{HoofdZaak}.
\paragraph{Poisson Lie algebra.} 
Let $C^{\infty}(X)$ be the Lie algebra of smooth functions on $X$, equipped with the Poisson bracket.
Then Theorem~\ref{noncptpoissonh2} shows that 
the central extensions of $C^{\infty}(X)$ by $\R$ are classified by
the compactly supported first de Rham cohomology of $X$,
	\be
	H^2(C^{\infty}(X), \R) \simeq H^1_{\mathrm{dR}, c}(X)\,.
	\ee
\paragraph{Hamiltonian vector fields.}
 Let $\ham(X)$ be the Lie algebra of hamiltonian vector fields.
 Then Theorem~\ref{H2Ham} shows that
 continuous central extensions of $\ham(X)$ by $\R$
 are classified by
\be\label{eq:hamvechom}
H^2(\ham(X),\R) \simeq Z^1_{c}(X)/d\Omega^{0}_{c,0}(X)\,,
\ee
where $Z^1_{c}(X)$ is the space of closed, compactly supported 1-forms  on $X$, and $\Omega^{0}_{c,0}(X)$ is the space of compactly supported smooth functions $f$ on $X$ that
integrate to zero, $\int_{X} f \vol = 0$. 
If $X$ is compact, then \eqref{eq:hamvechom} reduces to%this yields
\be \label{eq:incptcase}
H^2(\ham(X),\R) \simeq H^1_{\mathrm{dR}}(X)\,,
\ee
as one can 
always change $f$ by a suitable constant to make its integral zero. 
 If $X$ is not compact, then there is precisely one `extra' central extension.
 The 1-dimensional kernel of the canonical surjection 
 $Z^1_{c}(X)/d\Omega^{0}_{c,0}(X) \rightarrow H^1_{\mathrm{dR},c}(X)$ 
 is spanned by the Kostant-Souriau class $[\psi_{KS}]$, so that for noncompact $X$,
 the expression
\eqref{eq:hamvechom} reduces to 
\be
H^2(\ham(X),\R) \simeq H^1_{\mathrm{dR},c}(X) \oplus \R [\psi_{KS}]\,.
\ee
The class $[\psi_{KS}]$  
corresponds to the central extension 
 $\R \rightarrow C^{\infty}(X) \rightarrow \ham(X)$, which is trivial for compact $X$. 

The result \eqref{eq:incptcase} for compact $X$ was stated in \cite{Roger1995}, but 
we are not aware of any proof in the literature so far.
	   
\paragraph{Symplectic vector fields.}
To describe the continuous second Lie algebra cohomology of the Lie algebra 
$\mathrm{sp}(X)$ of symplectic vector fields, introduce the bilinear form
$H^1_{\mathrm{dR},c}(X) \times H^1_{\mathrm{dR}}(X) \rightarrow \R$, defined by
\[
([\alpha], [\beta]) := \int_{X}\alpha \wedge \beta \wedge \omega^{n-1}/(n-1)!,
\]
and the 4-linear form 
$H^1_{\mathrm{dR},c}(X) \times \wedge^3 H^1_{\mathrm{dR}}(X) \rightarrow \R$, defined by
\[
([\alpha], [\beta_1], [\beta_2], [\beta_3]) :=  \int_{X}\alpha \wedge 
\beta_1 \wedge \beta_2 \wedge \beta_3 \wedge \omega^{n-2}/(n-2)!.
\]

If $X$ is compact and connected, Theorem~\ref{thm:spcpt} shows that
\be\label{eq:SVsp1}
H^2(\mathrm{sp}(X),\R) \simeq {\textstyle \bigwedge^2}H^1_{\mathrm{dR}}(X)^*
\oplus \mathrm{K}_{\mathrm{c}},
\ee
%In equation \eqref{eq:SVsp1}, 
where $K_{\mathrm{c}} \subseteq H^1_{\mathrm{dR}}(X)$ 
is the set of classes $a \in H^1_{\mathrm{dR}}(X)$ such that
\[
(a, b_1, b_2, b_3) = 
\frac{1}{\mathrm{vol}(X)}
\sum_{\mathrm{cycl}}(a,b_1)(b_2,b_3)
\]
for all $b_1, b_2, b_3\in H^1_{\mathrm{dR}}(X)$.

If $X$ is noncompact and connected, Theorem \ref{thm:spncpt} shows that
\be\label{eq:SVsp2}
H^2(\smp(X),\R) \simeq {\textstyle \bigwedge^{2}}H^1_{\mathrm{dR}}(X)^*
\oplus \R [\psi'_{KS}]
\oplus \mathrm{K}_{\mathrm{nc}}\,,
\ee
where
$\psi'_{KS}(v,w)=\omega(v,w)_x$ is an extension to $\mathrm{sp}(X)$ of the
Kostant-Souriau cocycle $\psi_{KS}$,
and 
$\mathrm{K}_{\mathrm{nc}}
\subseteq H^1_{\mathrm{dR},c}(X)$ 
is the set of classes $a$ such that
$(a,b) = 0$ and $(a,b_1,b_2,b_3) = 0$
for all $b$ and $b_1,b_2,b_3$ in $H^1_{\mathrm{dR}}(X)$.

The case of compact, connected $X$ was treated in \cite{Vizman2006}
under the assumption that $H^2(\ham(X),\R) \simeq H^1_{\mathrm{dR}}(X)$, which is 
justified by the present paper.
The case of noncompact, connected $X$ appears to be new.

%\paragraph{Table with results}
%The results on second continuous Lie algebra cohomology are summarised as follows:
\begin{table}[h]
% title name of the table
\centering % centering table
\begin{tabular}{l c c rrrrrrr} % creating 10 columns
\hline \\[- 3 mm]%\hline % inserting double-line
\multicolumn{1}{c}{$\fg$} &\, & $X$ &\multicolumn{1}{c}{$H^2(\fg, \R)$}
\\ [0.5ex]
\hline \\[-3mm] % inserts single-line
% Entering 1st row
& &compact & $H^1_{\mathrm{dR}}(X)$  \\[-1ex]
\raisebox{1.5ex}{$C^\oo_c(X)$} & \raisebox{1.5ex}{\,}&noncompact
& $H^1_{\mathrm{dR}}(X)$  \\[1ex]
\hline\\[-3 mm]
% Entering 2nd row
& &compact & $H^1_{\mathrm{dR}}(X)$  \\[-1ex]
\raisebox{1.5ex}{$C^\oo(X)$} & \raisebox{1.5ex}{\,}& noncompact
&$H^1_{\mathrm{dR,c}}(X)$ \\[1ex]
\hline\\[-3 mm]
% Entering 3rd row
& &compact & $H^1_{\mathrm{dR}}(X)$  \\[-1ex]
\raisebox{1.5ex}{$\ham(X)$} & \raisebox{1.5ex}{\,}& noncompact
&$H^1_{\mathrm{dR,c}}(X)\oplus\R$  \\[1ex]
\hline\\[-3 mm]
% Entering 4th row
& &compact & $\La^2H^1_{\mathrm{dR}}(X)^*\oplus K_c$ \\[-1ex]
\raisebox{1.5ex}{$\mathrm{sp}(X)$} & \raisebox{1.5ex}{\,}& noncompact
&$\La^2H^1_{\mathrm{dR}}(X)^*\oplus K_{nc}\oplus\R$ \\[1ex]
% [1ex] adds vertical space
\hline % inserts single-line
\end{tabular}
\caption{Results on second continuous Lie algebra cohomology} 
\label{tab:PPer}
\end{table}

\subsubsection{The universal central extension} 

%\paragraph{The universal central extension.}

A Lie algebra possesses a universal central extension if and only if 
it is perfect \cite{vanderKallen1973}.
Using results in \cite{Lichnerowicz74},
we show in Corollary~\ref{cptsuppfirstcohomology},
Proposition~\ref{Prop:noncpspoisson}
and Corollary~\ref{hamp}
%that for a connected symplectic manifold $X$,
%only 
that of the above Lie algebras, only $\ham(X)$
and $C^{\infty}(X)$ are perfect; the former for any symplectic manifold, 
the latter only if $X$ has no compact connected components.

We describe the universal central extension of these two Lie algebras.
Let $* \,\colon \Omega^{k}(X) \rightarrow \Omega^{2n-k}(X)$ be
the 
\emph{symplectic Hodge star operator}, 
which is defined like the ordinary Hodge star operator, with
the Riemannian form replaced by the symplectic form.
Its defining property is that for all $\alpha, \beta \in \Omega^{k}(X)$,
$\beta \wedge *\alpha = ( \wedge^{k}\omega)(\beta,\alpha) \omega^{n}/n!$,
where $\wedge^{k}\omega$ is the $(-1)^{k}$-symmetric form induced by 
$\omega$ on $\wedge^{k}T^*X$.

% (cf.\ Section \ref{shsop}).
From the de Rham differential $d$, %$ \colon \Omega^{k}(X) \rightarrow \Omega^{k+1}(X)$,
one then obtains 
the \emph{canonical homology operator} $ \delta:= (-1)^{k+1}*d\,*$,
which lowers rather than raises the degree,
$\delta \colon \Omega^k(X) \rightarrow \Omega^{k-1}(X)$.
The universal central extension 
is the Fr\'echet Lie algebra
\be 
\Omega^1(X)/\delta\Omega^2(X)\,,
\ee
with Lie bracket
\be
[[\alpha],[\beta]] :=  [\delta\alpha \cdot d\delta\beta]\,. 
\ee
This is the 
universal central extension
of $\ham(X)$ for any connected, symplectic manifold $X$, and 
also of $C^{\infty}(X)$ if $X$ is noncompact.

\paragraph{Universal central extension for hamiltonian vector fields.}
 %of $\mathrm{\bf ham (}X\mathrm{\bf)}$}

To obtain the universal central extension of $\ham(X)$, equip $\Omega^1(X)/\delta\Omega^2(X)$
with the projection 
\be\label{univextintro}
\mathfrak{z} \longrightarrow \Omega^1(X)/\delta\Omega^2(X) \stackrel{\pi}{\longrightarrow} 
\ham(X)\,,
\ee
%$\pi \colon \Omega^1(X)/\delta\Omega^2(X) \rightarrow \ham(X)$
defined by 
$\pi([\alpha]) := X_{\delta\alpha}$, where $X_{f}$
is the hamiltonian vector field corresponding to $f \in C^{\infty}(X)$.
The kernel of $\pi$ 
is precisely the centre 
$
\mathfrak{z} = \mathrm{Ker}(d\circ\delta)/\delta\Omega^2(X)
$ of $\Omega^1(X)/\delta\Omega^2(X)$. 
If $X$ is compact, then $\mathfrak{z}$ is 
the first canonical homology 
\[
H^{\mathrm{can}}_{1}(X) := \mathrm{Ker}(\delta \colon \Omega^1(X) \rightarrow \Omega^0(X))/
 \mathrm{Im}(\delta \colon \Omega^2(X) \rightarrow \Omega^1(X))\,,
 \]
which is isomorphic to
the de Rham cohomology 
$H^{2n-1}_{\mathrm{dR}}(X)$, with $\mathrm{dim}(X) = 2n$,
 by the Hodge star operator.
 If $X$ is noncompact, then
$H^{2n-1}_{\mathrm{dR}}(X) \subset \mathfrak{z}$ is a subspace of co\-di\-men\-sion~1.

Theorem~\ref{universalHam} shows that for $X$ connected, the central extension
\eqref{univextintro} 
is universal for continuous extensions of $\ham(X)$ by a finite dimensional centre, and,
if $H^{2n-1}_{\mathrm{dR}}(X)$ is finitely generated,
even for 
linearly split continuous
extensions by an infinite dimensional centre. % if $H^{2n-1}_{\mathrm{dR}}(X)$ is finitely generated.
The requirement that a Lie algebra extension be linearly split is quite natural, and 
automatically satisfied if it comes from a Lie group extension.

\paragraph{Universal central extension for Poisson Lie algebra.}

For a noncompact, connected, symplectic manifold $X$, the universal central extension of 
$C^{\infty}(X)$ is
again $\Omega^{1}(X)/\delta\Omega^2(X)$, but with a different projection.
Corollary~\ref{Cor:univpois} shows that the central extension
\be \label{introcepois}
H^{\mathrm{can}}_{1}(X) \rightarrow \Omega^1(X)/\delta\Omega^2(X)\stackrel{\delta}{\longrightarrow} C^{\infty}(X)
\ee
is universal for continuous extensions of the Poisson Lie algebra $C^{\infty}(X)$ by a 
finite dimensional centre, and, 
if $H^{\mathrm{can}}_{1}(X)$ 
is finite dimensional, also
for 
linearly split continuous
extensions by an infinite dimensional centre. 

\paragraph{Roger cocycles and singular cocycles.} \label{RogerandSingular}
We describe two particularly relevant classes of 2-cocycles on $\ham(X)$,
related to de Rham cohomology and singular homology of $X$. 

By the universal property of (\ref{univextintro}) for a connected, symplectic manifold $X$,
%$\Omega^{1}(X)/\delta\Omega^2(X)$, 
continuous 2-cocycles $\psi_{}$ of $\ham(X)$
correspond bijectively to continuous linear functionals $S$
on $\Omega^{1}(X)/\delta\Omega^2(X)$
by
\begin{equation}\label{oneconvention}
\psi_{S}(X_{f},X_{g}) = S([fdg])\,.
\end{equation}

In order to describe the isomorphism \eqref{eq:hamvechom}, %$H^2(\ham(X),\R) \simeq Z^1_{c}(X)/d\Omega^0_{c,0}(X)$,
it is most convenient to consider the class of \emph{Roger cocycles} $\psi_{\alpha}$, 
related to de Rham cohomology of~$X$ \cite{Roger1995}.
They come from 
a compactly supported, closed 1-form $\alpha \in \Omega_{c}^1(X)$.
Via the functional $S_{\alpha}([\beta]) := \int_{X}*\beta \wedge \alpha$, they yield the
 2-cocycle
\be
\psi_{\alpha}(X_f,X_g)=S_\al([fdg]) = \int_{X} f(i_{X_g}\alpha) \vol\,.
\ee
The isomorphism $Z^1_{c}(X)/d\Omega^0_{c,0}(X) \simeq H^2(\ham(X),\R)$
is then given by $[\alpha] \mapsto [\psi_{\alpha}]$.

We also introduce a class of \emph{singular cocycles} $\psi_{N}$ that are expected to be 
of particular relevance in the projective unitary representation theory of $\ham(X)$.
They  come from $(2n-1)$-dimensional submanifolds $N$ of $X$. 
The functional $S_{N}([\beta]) = \int_{N} *\beta$
on $\Omega^1(X)/\delta\Omega^2(X)$ yields the 2-cocycle
\be
	\psi_{N}(X_{f}, X_{g}) =S_N([fdg])= \int_{N} fdg \wedge \omega^{n-1}/(n-1)!\,.
\ee

\subsection{Outline}
Having already summarised the main results of the present paper, we can be brief about the outline.

In Section \ref{s2}, we fix our notation and introduce the Hodge star operator. 
We describe the relevant notions of central extensions and Lie algebra cohomology, 
and prove some simple but useful lemmas.

In Section \ref{sec:chars}, we determine the characters of the Lie algebras associated to $X$,
which amounts to determining their first cohomology with trivial coefficients. 
We prove that $\ham(X)$ is perfect, as is 
 $C^{\infty}(X)$ for $X$ noncompact.

Section \ref{s4} is the heart of the paper. We prove that all continuous 2-cocycles are diagonal, 
and hence described by differential operators. We then show that they are of order at most 1.
We investigate how the cocycles of the different Lie algebras are 
related to each other. Starting from $C^{\infty}_{c}(X)$, we derive the specific form of the cocycles
for the Lie algebras associated to $X$.

In Section \ref{s5}, we give a detailed description of the universal central extension 
$\Omega^{1}(X)/\delta\Omega^{2}(X)$ of $\ham(X)$, and show that it also serves
as a universal central extension of $C^{\infty}(X)$ if $X$ is noncompact.

In Section \ref{s6}, we derive $H^2(\smp(X),\R)$ from the results in Section \ref{s4} using two similar, 
but slightly different transgression arguments, one 
for compact and one for noncompact manifolds $X$.

%%%%%%%%%%%%%%%%%%%%%%%%%%%%%%%%%%%

\section{Preliminaries}\label{s2} 
In this section, we collect some useful and essentially well known results, 
and adapt them to our setting. 
In Section \ref{sec:notconv}, we fix some notation.
In Section~\ref{sec:contcohomo},
we give the precise link between (universal) central extensions 
of locally convex Lie algebras
and continuous Lie algebra cohomology. In Section~\ref{shsop},
we describe the Hodge star operator and canonical cohomology
of symplectic manifolds. Finally, in Section~\ref{sec:somelems}, we collect some
basic lemmas that will be used later on.

\subsection{Notation and Conventions}\label{sec:notconv}

We fix some notation and conventions that will be used throughout the paper.
Let $(X,\omega)$ be a (second countable) symplectic manifold of dimension $2n$,
not necessarily connected unless specified otherwise.
Then the Lie algebra of \emph{symplectic vector fields} is denoted by
\[\mathrm{sp}(X) := \{v \in \mathrm{vec}(X) \,;\, L_{v}\omega = 0\}\,.\]
The \emph{hamiltonian vector field} $X_{f}$ of
$f \in C^{\infty}(X)$ is
the unique vector field such that $df = -i_{X_{f}} \omega$.
Since the Lie bracket of $v,w \in \mathrm{sp}(X)$
is hamiltonian with $-i_{[v,w]}\omega = d\omega(v,w)$,
the Lie algebra of hamiltonian vector fields
\[
\ham(X) := \{X_{f}\,;\, f \in C^{\infty}(X)\}
\]
is an ideal in $\mathrm{sp}(X)$.
The map $X \mapsto -[i_{X}\omega]$ identifies the
quotient of $\mathrm{sp}(X)$ by $\ham(X)$ with the abelian Lie algebra $H_{\mathrm{dR}}^{1}(X)$,
\begin{equation}
0 \rightarrow \ham(X) \rightarrow \mathrm{sp}(X) \rightarrow H_{\mathrm{dR}}^{1}(X) \rightarrow 0\,.
\end{equation}

We equip the space $C^{\infty}(X)$ of $\R$-valued, smooth  functions on $X$ 
with the \emph{Poisson bracket}
$\{f,g\} = \omega(X_{f},X_{g})$.
The kernel of the Lie algebra homomorphism $f \mapsto X_{f}$, the space 
$H^0_{\mathrm{dR}}(X)$ of locally constant functions, is central in $C^{\infty}(X)$.
We thus 
we obtain a central extension
\begin{equation}
0 \rightarrow H^{0}_{\mathrm{dR}}(X) \rightarrow C^{\infty}(X) \rightarrow \ham(X) \rightarrow 0\,.
\end{equation}

We also consider the subalgebra
$C^{\infty}_{c}(X)$ of
\emph{compactly supported} functions.
The map $C^{\infty}_{c}(X) \rightarrow H^{2n}_{\mathrm{dR},c}(X)$ defined by 
$f \mapsto [f \omega^n/n!]$ is a Lie algebra homomorphism into an abelian Lie algebra; every
commutator $\{f,g\} = i_{X_f}dg$ maps to zero, as
$(i_{X_{f}}dg)\omega^n/n! = dg \wedge (i_{X_{f}}\omega^n/n!) = d(g i_{X_{f}}\omega^n/n!)$ is exact.
We define the ideal $C^{\infty}_{c,0}(X)$ to be the kernel of this map,
\be\label{ide}
C^\oo_{c,0} (X) := \{f \in C_{c}^{\infty}(X)\,;\, 
0 = [f\omega^n/n!] \in H^{2n}_{\mathrm{dR},c}(X)\}\,,
%\exists 
%\ga\in \Omega_{c}^{2n-1}(X)\text{ s.t. } f \vol = d \ga
%\}\,.
\ee
so that we obtain an exact sequence 
\be
0 \rightarrow C^{\infty}_{c,0}(X) \rightarrow C^{\infty}_{c}(X) \rightarrow H^{2n}_{\mathrm{dR}, c}(X) \rightarrow 0\,.
\ee 
If $X$ is connected, $C^{\infty}_{c,0}(X)$ is the space of zero-integral functions, $\int_{X}f \vol = 0$.

\subsection{The symplectic hodge star operator}\label{shsop}

The \emph{symplectic Hodge star operator}
$* \,\colon \Omega^{k}(X) \rightarrow \Omega^{2n-k}(X)$ is uniquely defined 
by the requirement that
$\beta \wedge *\alpha =( \wedge^{k}\omega)(\beta,\alpha) \omega^{n}/n!$
for all $\beta \in \Omega^{k}(X)$,
where $\wedge^{k}\omega$ is the $(-1)^{k}$-symmetric form induced by 
$\omega$ on $\wedge^{k}T^*X$. 
Alternatively, it is described by
contraction 
with the tensor 
\[
T = \omega^{\mu_1\nu_1} \cdots \omega^{\mu_{k}\nu_{k}}\,
\partial_{\mu_{1}}\wedge \cdots \wedge \partial_{\mu_{k}} 
\otimes i_{\partial_{\nu_{k}}} \cdots \, i_{\partial_{\nu_{1}}} \vol.
\]
We define the \emph{canonical homology operator} 
$\delta \colon \Omega^k(X) \rightarrow \Omega^{k-1}(X)$
by
\be \label{canhomop}
\delta:= (-1)^{k+1}*d\,*\,.
\ee
In the following theorem, we gather some facts about $*$ and $\delta$
that we will need later.

\begin{Theorem}[Brylinski]\label{Brylinskisthm}
The symplectic star operator satisfies the equalities
\begin{eqnarray}
*f_0df_1\wedge\ldots\wedge df_k &=& (-1)^{k}f_0i_{X_{f_k}}\ldots i_{X_{f_1}}\omega^n/n! 
\quad \mathrm{and} \label{ster}\\
*^2 &=& \mathrm{Id}. \label{idempotentster}
\end{eqnarray}
Furthermore, if $\pi \in \Gamma(\wedge^{2}TX)$ is the Poisson bivector field defined by 
$\omega$, then $\delta  = i_{\pi} \circ d - d \circ i_{\pi}$. 
For all $f_0,\ldots,f_k \in C^{\infty}(X)$, we thus have
\begin{eqnarray}\label{PoissonCoh}
\delta (f_0 df_1\wedge \ldots \wedge df_k) &=& 
\sum_{i=1}^{k} (-1)^{i+1}\{f_0,f_i\}df_1\wedge \ldots d\hat{f}_{i} \ldots \wedge df_k \\ \nonumber
& + &
\sum_{1\leq i < j \leq k} (-1)^{i+j} f_0 d\{f_i,f_j\} df_1 \wedge \ldots d\hat{f}_{i} \ldots
d\hat{f}_{j} \ldots \wedge df_{k}\,. 
\end{eqnarray}
\end{Theorem}
\begin{proof}
Equations~\eqref{idempotentster} and~\eqref{PoissonCoh} are proven in
\cite[\S 2]{Br88},  and equation~\eqref{ster} in \cite[\S 5]{Brylinski1990}.
\end{proof}
% In Brylinski's paper, the order of the vector fields is reversed and i_{X_{f}}\omega = +df, which accounts for
% the minus sign.

\begin{Remark}\label{kis12}
The special cases $k = 1$ and $k=2$ of \eqref{PoissonCoh} will be of special use later on:
\begin{gather}
\de(f_0df_1)=\{f_0,f_1\}\,,\label{rkform1}\\
\de(f_0df_1\wedge df_2)=\{f_0,f_1\}df_2-\{f_0,f_2\}df_1-f_0d\{f_1,f_2\}\,. \label{rkform2}
\end{gather} 
\end{Remark}

\begin{Lemma}\label{lem1}
$d\Omega^0(X) \subseteq \delta \Omega^{2}(X)$.
\end{Lemma}

\begin{proof}
For $f \in \Omega^0(X)$, 
we have $*df = df\wedge \omega^{n-1}/(n-1)!$ by eqn.~\eqref{ster},
so with $*^2 = \mathrm{Id}$, we find
$df = * d(f \omega^{n-1}/(n-1)!)$. 
Using $\delta = (-1)^{k+1}*d\,*$, we then find that 
$df = -\delta (* f\omega^{n-1}/(n-1)!)$ is in $\delta\Omega^2(X)$ as desired.
\end{proof}

We define the compactly supported \emph{canonical homology} 
$H^{\can}_{c,\bullet}(X)$ to be the homology of the complex
$(\Omega^{\bullet}_{c}(X), \delta)$.
By definition then, the star operator 
$*\, \colon \Omega^{k}_{c}(X) \rightarrow \Omega^{2n-k}_{c}(X)$, 
adorned with the sign $(-1)^{k(k-1)/2}$, defines an
isomorphism of chain complexes
\begin{center}
\mbox{
\xymatrix
{
%0 \ar[r] &
\Omega^0_{c}(X) \ar[r]^-{d}&  
\Omega^1_{c}(X)\ar[r]^-{d} &
\quad \ldots \quad \ar[r]^-{d}&
\Omega^{2n-1}_{c}(X)\ar[r]^-{d}&
\Omega^{2n}_{c}(X)
%\ar[r]^-{d}&0
\\
%0 \ar[r] &
\Omega^{2n}_{c}(X) \ar[r]^-{\delta}\ar@{<->}[u]^{*}&  
\Omega^{2n -1}_{c}(X)\ar[r]^-{\delta}\ar@{<->}[u]^{*} &
\quad \ldots \quad\ar[r]^-{\delta}&
\Omega^{1}_{c}(X)\ar[r]^-{\delta}\ar@{<->}[u]^{*}&
\Omega^{0}_{c}(X),%\ar[r]^-{\delta}
\ar@{<->}[u]^{*}
%& 0
}
}
\end{center}
hence an isomorphism 
$H^{k}_{\mathrm{dR}, c}(X) \simeq H^{\can}_{c,2n-k}(X)$.
Mutatis mutandis, the same argument in the noncompactly supported 
case yields an isomorphism 
$H^{k}_{\mathrm{dR}}(X) \simeq H^{\can}_{2n-k}(X)$.

\begin{Proposition}\label{Prop:candRiso}
The symplectic Hodge star operator yields isomorphisms 
$H^{k}_{\mathrm{dR}}(X) \simeq H^{\can}_{2n-k}(X)$ and
$H^{k}_{\mathrm{dR}, c}(X) \simeq H^{\can}_{c,2n-k}(X)$.
\end{Proposition}

\begin{Remark}
For an arbitrary Poisson manifold, the canonical homology $H_{\bullet}^{\mathrm{can}}(X)$
is defined as the homology of the complex $\Omega^{\bullet}(X)$
with respect to the differential $\delta  = i_{\pi} \circ d - d \circ i_{\pi}$.
Brylinski's theorem~\ref{Brylinskisthm} then shows that in case $X$
is symplectic, $H_{k}^{\mathrm{can}}(X) \simeq H^{2n -k}_{\mathrm{dR}}(X)$.
\end{Remark}

\subsection{Lie algebra cohomology}\label{sec:contcohomo}

The objective of this paper is to classify central extensions of certain 
locally convex Lie algebras.

\paragraph{Locally convex Lie algebras.}
A \emph{locally convex Lie algebra} is a locally convex 
topological vector space $\mathfrak{g}$, together with a continuous Lie bracket
$[\,\cdot\,,\,\cdot\,]\colon \mathfrak{g} \times \mathfrak{g} \rightarrow \mathfrak{g}$.

Let $(X,\omega)$ be a symplectic manifold, which we will always assume to be second countable. 
If $K\subseteq X$ is a compact subset with open interior, then
\[C_{K}^{\infty}(X) := \{f \in C^{\infty}(X)\,;\, \mathrm{supp}(f) \subseteq K\}\]
and 
\[\mathrm{vec}_{K}(X) := \{v \in \mathrm{vec}(X)\,;\, \mathrm{supp}(v) \subseteq K\}\]
are locally convex Lie algebras for the Fr\'echet topology of uniform convergence in all derivatives.
We equip $C^{\infty}(X)$ and $\mathrm{vec}(X)$ with the Fr\'echet topology
that comes from the inverse limit over $K\subseteq X$, and
$C^{\infty}_{c}(X)$ with the LF-topology that comes from the (strict)
direct limit. 
The former makes $\mathrm{sp}(X)$, $\ham(X)$ and $C^{\infty}(X)$ into 
locally convex (Fr\'echet) Lie algebras, and the latter makes 
$C^{\infty}_{c}(X)$ and $C^\oo_{c,0}(X)$ into locally convex (LF) Lie algebras \cite[\S I.13]{Treves1967}.

\begin{Definition}\label{Extensions}
A \emph{continuous central extension} $(\hfg, \iota, \pi)$ of a locally 
convex Lie algebra $\mathfrak{g}$ by 
a locally convex vector space
$\mathfrak{a}$ 
%$\R$
is a continuous exact sequence
\[
0\rightarrow 
%\R
\mathfrak{a} 
\stackrel{\iota}{\longrightarrow} \widehat{\mathfrak{g}} 
\stackrel{\pi}{\longrightarrow} \mathfrak{g} \rightarrow 0
\]
of locally convex Lie algebras
such that 
$%\iota(\R) 
\iota(\mathfrak{a})
\subseteq \widehat{\mathfrak{g}}$ is central.
It is called \emph{linearly split} if there exists
a continuous linear map $\sigma \colon \fg \rightarrow \widehat{\fg}$ 
such that $\pi \circ \sigma = \mathrm{Id}_{\fg}$, and \emph{split} if $\sigma$
is a Lie algebra homomorphism.
An \emph{isomorphism} between
$(\widehat{\mathfrak{g}}, \iota, \pi)$ and $(\widehat{\mathfrak{g}}', \iota', \pi')$
is a continuous 
Lie algebra isomorphism $\phi \colon \widehat{\mathfrak{g}} \rightarrow \widehat{\mathfrak{g}}'$
such that $\phi \circ \pi' = \pi$ and $\phi \circ \iota = \iota'$.
\end{Definition}

\begin{Remark}\label{Rk:HahnBanach}
Continuous central extensions of $\fg$ by $\R$ are automatically linearly split. 
%by the Hahn-Banach Theorem.
Indeed, the map $\iota^{-1} \colon \iota(\R) \rightarrow \R$ extends to a
continuous linear functional
$\gamma \colon \hfg \rightarrow \R$ by the Hahn-Banach theorem for locally convex 
vector spaces \cite[Thm.\ 3.6]{Rudin1991},
and $\sigma$ is the inverse of $\pi \colon {\mathrm{Ker}(\gamma) \rightarrow \fg}$.
%(In concrete situations, one can often bypass the Hahn-Banach theorem 
%by choosing an explicit splitting.)
\end{Remark}

Linearly split continuous central extensions 
$\mathfrak{a} \rightarrow \widehat{\fg} \rightarrow \fg$ are the 
infinitesimal counterpart of
central extensions $A \rightarrow \widehat{G} \rightarrow G$ 
of infinite dimensional Lie groups modelled on 
locally convex spaces
\cite{Milnor1984, Neeb2006}.
Continuity of the Lie algebra homomorphisms comes from smoothness
of the group homomorphisms, and a splitting of $\widehat{\fg} \rightarrow \fg$ 
exists because, by assumption, $\widehat{G} \rightarrow G$ is 
a locally trivial principal bundle \cite{Neeb03}.

Central extensions are classified by second Lie algebra cohomology. 
In order to classify \emph{continuous} central extensions, we 
will need \emph{continuous} Lie algebra cohomology.

\begin{Definition} \label{Cohomology}
The \emph{continuous Lie algebra
cohomology} $H^n(\mathfrak{g},\R)$ of a locally convex Lie algebra $\mathfrak{g}$ 
is the cohomology of the complex
$C^{\bullet}(\mathfrak{g},\R)$, where $C^n(\mathfrak{g},\R)$ is the space of 
continuous alternating $n$-linear
maps $\psi \colon \fg^n \rightarrow \R$ 
with differential $\dd \colon C^{n}(\mathfrak{g},\R) \rightarrow C^{n+1}(\mathfrak{g},\R)$ defined by zero on 
$C^{0}(\fg, \R)$, and
\[\dd \psi(x_0,\ldots,x_{n}):= \sum_{0\leq i<j\leq n}
(-1)^{i+j} \omega([x_i,x_j],x_0,\ldots,\widehat{x}_i, \ldots, 
\widehat{x}_j, \ldots, x_n)\]
for $n\geq 1$.
\end{Definition}

The cohomology in degree 1 classifies the continuous characters of $\fg$,
that is, the continuous Lie algebra
homomorphisms $\fg \rightarrow \R$.
Indeed, $H^1(\fg,\R)$ is the topological dual of the abelian Lie algebra 
$(\fg/\overline{[\fg,\fg]})$, where $\overline{[\fg,\fg]}$ is the 
closure of the commutator ideal. 
In particular, a locally convex Lie algebra is 
topologically perfect ($\fg = \overline{[\fg,\fg]}$)
if and only if $H^1(\fg,\R)$ vanishes.

%The cohomology in degree 2 classifies the continuous central extensions of $\fg$.

The following standard result (cf.~\cite{JanssensNeeb2015})
interprets $H^2(\fg, \R)$ in terms of central extensions.
\begin{Proposition}\label{extensionvscohomology} Up to isomorphism, 
the continuous central extensions of the locally convex Lie algebra $\fg$ are classified by $H^2(\fg,\R)$.
\end{Proposition}
\begin{proof}
Given a 2-cocycle $\psi \colon \fg^2 \rightarrow \R$, 
we define the Lie algebra $\widehat{\mathfrak{g}}_{\psi}$ by
\[
\widehat{\mathfrak{g}}_{\psi} :=  \R \oplus_{\psi} \fg
\]
with the Lie bracket
$[(z,x), (z',x')] := \big(\psi(x,x'), [x,x']\big)$. 
Equipped with the obvious maps
$\R \rightarrow \widehat{\mathfrak{g}}_{\psi} \rightarrow \mathfrak{g}$, this
is a continuous central extension 
of $\mathfrak{g}$ by $\R$, and a cohomologous cocycle $\psi' = \psi + \dd \lambda$ 
gives rise to an isomorphic Lie algebra $\widehat{\fg}_{\psi'}$, where
the
isomorphism $\widehat{\mathfrak{g}}_{\psi} \rightarrow \widehat{\mathfrak{g}}_{\psi'}$
is given by
$(z,x) \mapsto (z-\lambda(x),x)$.
Conversely, every continuous central extension $(\widehat{\fg},\iota,\pi)$ 
%$\R \rightarrow \widehat{\mathfrak{g}} \rightarrow \fg$
admits a continuous linear splitting $\sigma \colon \fg \rightarrow \widehat{\mathfrak{g}}$
with $\pi \circ \sigma = \mathrm{Id}_{\fg}$ by the Hahn-Banach Theorem, 
cf.\ Remark~\ref{Rk:HahnBanach}.
The continuous central extension $\R \rightarrow \widehat{\fg} \rightarrow \fg$ is thus 
isomorphic to 
$\R \rightarrow \widehat{\fg}_{\psi}\rightarrow \fg$ 
with 
\be\label{psis}\psi(x,x') := \gamma \Big([\sigma(x),\sigma(x')] - \sigma([x,x'])\Big)\,.\ee
The isomorphism $\widehat{\fg}_{\psi} \rightarrow \widehat{\fg}$
is given by $(z,x) \mapsto \sigma(x) + \iota(z)$.
\end{proof}

If $(\widehat{\fg}, \iota,\pi)$ is split, there exists a
continuous character
$\lambda \colon \widehat{\fg} \rightarrow \R$ (a \emph{splitting}) 
such that $\lambda \circ \iota = \mathrm{Id}_{\R}$.
The inverse of $\pi \colon \mathrm{Ker}(\lambda) \rightarrow \fg$ then yields a 
continuous Lie algebra homomorphism $\sigma \colon \fg \rightarrow \hfg$, so that
$(\widehat{\fg}, \iota,\pi)$ is isomorphic to the trivial extension, and its
class in $H^2(\fg,\R)$ is zero.
The space \[\{\lambda \in \widehat{\fg}'\,;\, \dd \lambda = 0, 
 \lambda \circ \iota = \mathrm{Id}_{\R} \}\]
of continuous splittings is then an affine subspace of $H^1(\widehat{\fg}, \R)$.

\begin{Definition}
If $\fg$ is a locally convex Lie algebra, and 
$\mathfrak{a}$ a locally convex space, then
a continuous central extension 
\[
\mathfrak{z} \rightarrow 
\widehat{\fg} \rightarrow \fg
\]
of $\fg$
is called
$\mathfrak{a}$-\emph{universal} 
%for a
%locally convex space
%$\mathfrak{a}$ 
if for every linearly split central extension
$\mathfrak{a} \rightarrow \fg^{\sharp}\rightarrow \fg$ of $\fg$ by 
$\mathfrak{a}$, there exists a unique continuous Lie algebra homomorphism 
$\phi \colon \widehat{\fg} \rightarrow \fg^{\sharp}$ such that
the following diagram is commutative:
%\begin{center}
%\begin{tikzcd}\mathfrak{z} 
%\arrow[d,"\phi|_{\mathfrak{z}}"] \arrow[r] & \widehat{\fg} \arrow[d,"\phi"]\arrow[r] & 
%\fg \arrow[d,equal]\\\mathfrak{a} \arrow[r] & \fg^{\sharp} \arrow[r] & \,\,\fg \,.              
%\end{tikzcd}
%\end{center}
\begin{center}
\mbox{
\xymatrix
{
\mathfrak{z} \ar[r]&  
\widehat{\fg}\ar[r]&
\,\fg\,
\\
\mathfrak{a} \ar[r]\ar@{<-}[u]_{\phi|_{\mathfrak{z}}}&  
\fg^{\sharp}\ar[r]\ar@{<-}[u]_{\phi} &
\,\fg.%\ar[r]^-{\dd}
\ar@{<-}[u]_{\mathrm{Id}}
}
}
\end{center}
A central extension is called \emph{universal} if it is linearly split and $\mathfrak{a}$-universal
for every locally convex space $\mathfrak{a}$.
\end{Definition}
If a universal central extension exists, then it is 
unique up to continuous isomorphisms. 
A sufficient condition for existence is that
$\fg$ be a perfect Fr\'echet Lie algebra with finitely generated $H^{2}(\fg,\R)$ 
(cf.~\cite[Cor.~2.13]{Neeb03}). %, then it possesses a universal central extension. 

%%%%%%%%%%%%%%%%%%

\subsection{Some useful lemmas}\label{sec:somelems}

We close this introductory section with two small lemmas that will be useful throughout the text.

\begin{Lemma}\label{imde}
If $X$ is noncompact, then $\de:\Om^1(X)\to C^\oo(X)$ is surjective.
If $X$ is compact, then $\mathrm{Im}(\delta) = C^\oo_{c,0}(X)$.
\end{Lemma}

\begin{proof}
 The cokernel of $\delta \colon \Omega^{1}(X) \rightarrow C^{\infty}(X)$ 
 is by definition $H_{0}^{\mathrm{can}}(X)$, hence isomorphic to $H^{2n}_{\mathrm{dR}}(X)$  (cf.\ Prop.~\ref{Prop:candRiso}).
 If $X$ is noncompact, then $H^{2n}_{\mathrm{dR}}(X)$ vanishes, so $\delta$ is surjective.
If $X$ is compact, the inclusion
$\mathrm{Im}(\delta) \subseteq C^\oo_{c,0}(X)$
holds because for all $\alpha \in \Omega^{1}(X)$, we have
$\int_{X} (\delta \alpha) \omega^{n}/n! = \int_{X}*\delta\alpha = \int_{X} d*\alpha = 0$.
Equality follows because the cokernel $H^{\mathrm{can}}_{0}(X) \simeq H^{2n}_{\mathrm{dR}}(X)$ of $\delta$
is 1-dimensional.
\end{proof}

\begin{Lemma}\label{lem:4termcyc}
For $\alpha \in \Omega^1(X)$ and $v_1, v_2, v_3$ and $v$ in $\mathrm{vec}(X)$, we have
\begin{eqnarray}
\alpha(v)\omega^n/n! &=& \alpha \wedge i_{v}\omega\wedge \omega^{n-1}/(n-1)!\,,
\label{kleintje}\\
\sum_{\mathrm{cycl}} \alpha(v_1)\omega(v_2,v_3)\omega^n/n! &=&
\alpha \wedge i_{v_1}\omega \wedge i_{v_2}\omega \wedge i_{v_3}\omega \wedge \omega^{n-2}/(n-2)!\,.\label{forma}
\end{eqnarray}
\end{Lemma}
\begin{proof}
The proof is a simple computation. 
Equation~\eqref{kleintje} is obtained from
expanding $i_v (\alpha \wedge \omega^n)/n! = 0$.
Applying this to
$\alpha = -i_{w}\omega$, we find
\be\label{zoethout}
\omega(v,w)\omega^{n}/n! = i_{v}\omega \wedge i_{w}\omega \wedge \omega^{n-1}/n!\,.
\ee
Equation \eqref{forma} is then derived by expanding 
\[
0 = i_{v_3}(\alpha \wedge i_{v_1} \omega\wedge i_{v_2}\omega\wedge \omega^{n-1}/(n-1)!)\]
into an alternating sum of four terms, and applying \eqref{kleintje} and \eqref{zoethout}
repeatedly.
\end{proof}

%%%%%%%%%%%%%%%%%%

\section{Characters} \label{sec:chars}
 
In order to classify the continuous central extensions of the Poisson Lie algebra 
and the Lie algebra of Hamiltonian vector fields
(which, by Prop.~\ref{extensionvscohomology}, is equivalent to determining the 
second Lie algebra cohomology), we 
will need the first continuous Lie algebra cohomology.
This is equivalent to finding 
the closure of the commutator ideal, or, equivalently, the set of continuous characters 
$\fg \rightarrow \R$. 

\subsection{Characters of $C^\oo_{c,0}(X)$ and $C^{\infty}_{c}(X)$}

For the following proposition and its proof, we follow closely the paper \cite{Lichnerowicz74} of 
Avez, Lichnerowicz and Diaz-Miranda.

Define the linear functional
\be\label{lam}
\lambda \colon C^{\infty}_{c}(X) \rightarrow \R,
\quad \lambda(f) := \int_{X} f\vol, 
\ee
and note that the ideal defined in \eqref{ide} is $C^\oo_{c,0}(X) = \mathrm{Ker}(\lambda)$. 

\begin{Proposition}{\rm (\cite[\S 12]{Lichnerowicz74})}   \label{ALDM}
The commutator ideal of $C^{\infty}_{c}(X)$ is the perfect Lie algebra $C^\oo_{c,0}(X)$,
$$
[C^{\infty}_{c}(X), C^{\infty}_{c}(X)] = [C^\oo_{c,0}(X),C^\oo_{c,0}(X)] = C^\oo_{c,0}(X)\,.
$$
\end{Proposition}

\begin{proof}
Suppose $f \in C_{c}^{\infty}(X)$ is a commutator, 
$f = \{g,h\}$ for $g, h \in C_{c}^{\infty}(X)$. 
We show that there exist 
$g_0, h_0 \in C^\oo_{c,0}(X)$
with $f = \{g_0,h_0\}$.
Choose 
$\chi \in C_{c}^{\infty}(X)$
which is constant on $\mathrm{supp}(g) \cup \mathrm{supp}(h)$ and such that 
$\lambda(\chi) = 1$, with $\la$ defined in \eqref{lam}.   
The functions 
$g_0 := g - \lambda(g)\chi$ and $h_0 := h - \lambda(h)\chi$
are then the elements of $C^\oo_{c,0}(X)$ with the desired relation
$\{g_0,h_0\} = \{g,h\}$, because $\{g, \chi\} = \{h, \chi\} = 0$.
By applying this to (finite) sums of commutators, we see that 
$[C^{\infty}_{c}(X), C^{\infty}_{c}(X)]$ equals $[C^\oo_{c,0}(X),C^\oo_{c,0}(X)]$.

Because 
$L_{X_g}\omega = 0$ and $\{g,h\} = L_{X_{g}}h$, we have 
$f  \vol = L_{X_{g}}(h\vol)$, and hence
$f \vol = d(h i_{X_g}\vol)$. This shows that 
$f = \{g,h\}$ is an element of $C^\oo_{c,0}(X)$.
Since this argument extends to finite sums of commutators, we find
$$
[C^{\infty}_{c}(X), C^{\infty}_{c}(X)] = [C^\oo_{c,0}(X),C^\oo_{c,0}(X)] \subseteq C^\oo_{c,0}(X)\,,
$$
and we record the useful equation
\begin{eqnarray}
\{g,h\} \vol &=& d(h i_{X_g}\vol)\nonumber\\
&=& -d(h  dg \wedge \omega^{n-1}/(n-1)!)\nonumber\\
&=& dg \wedge dh \wedge \omega^{n-1}/(n-1)! \label{commutatorexact}
\end{eqnarray} 
for later use.

For the converse inclusion,
suppose that $f\vol = d\psi$ with $\psi$ compactly supported.
We show that $X_f$ is in the commutator ideal.
Write $\psi = 
\sum_{k=1}^{m} \psi_{k}$, where $\psi_{k}$ has compact
support in an area with Darboux coordinates $x^i,p^i$.
Note that $dx^i \wedge \omega^{n-1}/(n-1)!$ and 
$dp^i \wedge \omega^{n-1}/(n-1)!$ constitute a basis for
$\wedge^{2n-1}TX_{x}$ at each point $x\in X$, so that we can write
$$\psi_k = \sum_{i=1}^{n} \phi_{k}^{i} dx^i \wedge \omega^{n-1}/(n-1)!
+ \chi_{k}^{i} dp^i \wedge \omega^{n-1}/(n-1)!\,,$$ with 
$\phi_{k}^{i}$ and $\chi_{k}^{i}$ compactly supported.
Then choose compactly supported 
$\xi_{k}^{i}$ and $\eta_{k}^{i}$ 
that equal $x^i$ and $p^i$ on the support 
of $\phi_k^{i}$ and $\chi_k^{i}$ respectively to obtain
$$\psi_k = \sum_{i=1}^{n} \phi_{k}^{i} d\xi_k^i \wedge \omega^{n-1}/(n-1)!
+ \chi_{k}^{i} d\eta_{k}^i \wedge \omega^{n-1}/(n-1)!\,,$$
and thus 
$$d\psi = \sum_{k=1}^{m}\sum_{i=1}^{n} d\phi_{k}^{i} \wedge d\xi_k^i \wedge \omega^{n-1}/(n-1)!
+ d\chi_{k}^{i} \wedge d\eta_{k}^i \wedge \omega^{n-1}/(n-1)!\,.$$
By eqn.~(\ref{commutatorexact}), $d\psi = f\omega^{n}/n!$ then implies
$$f = \sum_{i=1}^{n} 
\sum_{k=1}^{m} \{\phi_k^i , \xi_k^i\} + \{\chi_k^i , \eta_k^i\}\,.$$
This shows the inverse inclusion 
$
[C^{\infty}_{c}(X), C^{\infty}_{c}(X)]  \supseteq C^\oo_{c,0}(X)
$. 
\end{proof}

Since $C^\oo_{c,0}(X)$ is perfect, it is in particular topologically perfect, so its first continuous Lie algebra 
cohomology vanishes.
Since $C^\oo_{c,0}(X)$ is closed in $C^{\infty}_{c}(X)$, the first cohomology  
$H^{1}(C_{c}^{\infty}(X),\mathbb{R})$
is equal to $(C_{c}^{\infty}(X) / C^\oo_{c,0}(X))'$.
The volume form $\vol$ then yields an isomorphism 
with the topological dual
$H^{2n}_{\mathrm{dR},c}(X)' = (\Omega_{c}^{2n}(X,\R)/d\Omega_{c}^{2n-1}(X,\R))'$
of the compactly supported de Rham cohomology,
taking $\gamma \in H^{2n}_{\mathrm{dR},c}(X)'$ to the class
of the cocycle $\lambda_{\gamma}(f) = \gamma([f\vol])$.
If $X$ has finitely many connected components, then the map 
$H_{2n}(X,\R) \rightarrow H^{2n}_{\mathrm{dR},c}(X)'$
defined by $[X_0] \mapsto \gamma_{X_0}$ with $\gamma_{X_0}([\beta]) := \int_{[X_0]} \beta$
is an isomorphism.
%(By the Universal Coefficient Theorem, $H_{2n}(X,\R) \simeq H_{2n}(X,\Z)\otimes_{\Z}\R$.)
The image of the integral singular cohomology $H_{2n}(X,\Z)$ then yields 
a lattice in $H^1(C^{\infty}_{c}(X),\R)$ with generators
$[\lambda_{X_0}]$, where
\begin{equation}\label{eencocykel}
\lambda_{X_0}(f) := \int_{X_0} f\omega^{n}/n!
\end{equation}
and $[X_0]$ runs over the connected components of $X$.

\begin{Corollary}\label{cptsuppfirstcohomology}
The continuous first Lie algebra cohomology $H^1(C^\oo_{c,0}(X), \R)$ vanishes.
Furthermore,
$H^{1}(C_{c}^{\infty}(X),\mathbb{R})$ is 
isomorphic to 
%the topological dual 
$H^{2n}_{\mathrm{dR},c}(X)'$.
If $X$ has finitely many connected components, 
then this is isomorphic to $H_{2n}(X,\R)$.
%$H_{2n}(X,\Z)\otimes_{\Z}\R$.
%of the 
%compactly supported de Rham cohomology. 
\end{Corollary}

\subsection{Characters of $C^{\infty}(X)$}

We turn to the Poisson Lie algebra $C^{\infty}(X)$ without compact support condition.
It is the direct product 
\[C^{\infty}(X) = \prod C^{\infty}(X_{x}),\] 
where the product runs over the connected 
components $X_{x}$ of $X$. If $X_{x}$ is compact, then
by Corr.~\ref{cptsuppfirstcohomology}, the pullback of the cocycle $\lambda_{X_{x}}$
by the canonical projection $C^{\infty}(X) \rightarrow C^{\infty}(X_{x})$
contributes to the first cohomology.
The following proposition says that there is no such contribution if $X_{x}$ is noncompact. 

\begin{Proposition}\label{Prop:noncpspoisson}
The first Lie algebra cohomology $H^1(C^{\infty}(X),\R)$ of the Poisson Lie algebra
$C^{\infty}(X)$ is isomorphic to $H_{2n}(X,\R)$.
%$H_{2n}(X,\Z)\otimes \R$.
Moreover, $C^{\infty}(X)$ is perfect, 
$C^{\infty}(X) = [C^{\infty}(X), C^{\infty}(X)]$,
if and only if $X$ has no compact connected components.
\end{Proposition}
\begin{proof}
Let $X = X_{\mathrm{cpt}} \sqcup X_{\mathrm{ncpt}}$, with $X_{\mathrm{cpt}}$
the union of the compact connected components $X_{x}$, 
and $X_{\mathrm{ncpt}} := X - X_{\mathrm{cpt}}$.
By Corrolary~\ref{cptsuppfirstcohomology}, any continuous character 
$\lambda \colon C^{\infty}(X) \rightarrow \R$
restricts to a multiple of $\lambda_{X_{x}}$ (cf.\ eqn.~(\ref{eencocykel})) 
on $X_{x}$. By continuity, $\lambda$ must then be a 
\emph{finite} linear combination of $\lambda_{X_{x}}$ on $C^{\infty}(X_{\mathrm{cpt}},\R)$.
Since $C^\oo_{c,0}(X_{\mathrm{ncpt}})$ is dense in $C^{\infty}(X_{\mathrm{ncpt}})$ for the topology of uniform 
convergence of all derivatives on compact subsets, it follows from
Prop.~\ref{ALDM} that $C^{\infty}(X_{\mathrm{ncpt}})$ is topologically perfect, 
so that the linear map 
$$
H_{2n}(X,\R)
%H_{2n}(X,\Z)\otimes_{\Z}\R 
\rightarrow H^1(C^{\infty}(X),\R)
$$
defined by $[X_{x}] \mapsto \lambda_{X_{x}}$
is an isomorphism.
All that is left to show is that $C^{\infty}(X)$ is perfect 
(rather than merely \emph{topologically} perfect)
if $X$ has no compact connected
components.
For this, we use the Brouwer-Lebesgue `Paving Principle'.
%(cf.~\cite{Amemiya75}). 
Since $H^{2n}_{\mathrm{dR}}(X) = 0$,
%If $X$ has no compact connected components, 
%connected and noncompact, 
there exists a $\psi \in \Omega^{2n-1}(X)$ 
such that $f\omega^n = d\psi$. Find a cover $U_{k,r}$ of $X$ where
$k \in \mathbb{N}$ is a countable index, $r \in \{1,\ldots,2n+1\}$ is a finite index,
and $U_{k,r} \cap U_{k',r} = \emptyset$ for all $k\neq k'$.
(Such a cover exists \cite[Thm.~V1]{Hurewicz41}.)
As in the proof of Prop.~\ref{ALDM}, we can write 
$\psi = \sum_{k=1}^{\infty}\sum_{r=1}^{2n+1}\psi_{k,r}$
with $\psi_{k,r}$ supported in $U_{k,r}$.
(Evaluated in a single point, the sum has at most $2n+1$ nonzero terms.)
We define $f_{k,r}$ by $d\psi_{k,r} = f_{k,r}\,\omega^{n}/n!$, and 
follow the proof of Prop.~\ref{ALDM} to
find $2n$ functions $g^{i}_{k,r}$ and $h^{i}_{k,r}$, supported in $U_{k,r}$,
that satisfy $f_{k,r} = \sum_{i=1}^{2n} \{g_{k,r}^{i},h_{k,r}^{i}\}$.
We set $G_{r}^{i} := \sum_{k=1}^{\infty}g_{k,r}^{i}$
and $H_{r}^{i} := \sum_{k=1}^{\infty}h_{k,r}^{i}$, and use the fact that 
$\{g^{i}_{k,r}, h^{i}_{k',r}\} = 0$ for $k \neq k'$ to find
\[\sum_{r=1}^{2n+1}\sum_{i=1}^{2n} \{G_{r}^{i},H_{r}^{i}\} =
%\sum_{r=1}^{2n+1}\sum_{i=1}^{2n}\sum_{k=1}^{\infty}\sum_{k'=1}^{\infty}
%\{g_{k,r}^{i},h_{k',r}^{i}\}
%=
\sum_{r=1}^{2n+1}\sum_{i=1}^{2n}\sum_{k=1}^{\infty}
\{g_{k,r}^{i},h_{k,r}^{i}\} 
=
\sum_{r=1}^{2n+1}
\sum_{k=1}^{\infty}
f_{k,r}
= f\,.
\] 
\end{proof}
\begin{Remark}\label{Rk:numbercomm}
Note that if $X$ has no compact connected components, then $f \in C^{\infty}(X)$ can be 
written as a sum of at most $2n(2n +1)$ commutators. 
By the above reasoning, the same holds for $f \in C^\oo_{c,0}(X)$
if $X$ is connected and compact. 
\end{Remark}

We summarise the classification of the characters of $C^\oo_{c,0}(X)$, $C^{\infty}_{c}(X)$
and $C^{\infty}(X)$.
Since $C^\oo_{c,0}(X)$ is perfect %(rather than merely topologically perfect) 
by Prop.~\ref{ALDM},
it has no nontrivial characters, continuous or not. 
If $X_{x}$ is a connected component of $X$, then 
$\lambda_{X_{x}}(f) = \int_{X_{x}}f \vol$
defines a character on $C^{\infty}_{c}(X)$. 
If we specify a real number $a_{x}\in \R$ 
for every connected component $X_{x}\subseteq X$,
%is nonzero for \emph{countably} many connected components $X_{x}\subseteq X$, 
then
\begin{equation}\label{eq:charcoal}
\lambda = \sum a_{x} \lambda_{X_{x}}
\end{equation}
is a character of $C^{\infty}_{c}(X)$. 
%, and if $X$ has finitely many connected components, then every character is of this form.
If $X_{x}$ 
is compact, then $\lambda_{X_{x}}$ extends to a character of $C^{\infty}(X)$, and the characters of 
$C^{\infty}(X)$ are \emph{finite} linear combinations of the $\lambda_{X_{x}}$.
The characters of $C^{\infty}(X)$ and $C^{\infty}_{c}(X)$ are automatically continuous.

\subsection{The  Kostant-Souriau extension}
If $X$ is a connected symplectic manifold, then the Lie algebra homomorphism
$\pi \colon C^{\infty}(X) \rightarrow \ham(X)$ defined by $f \mapsto X_{f}$ has kernel $\R \one$.
The continuous central extension
\be\label{eq:KSextension} 
\R \one \stackrel{\iota}{\longrightarrow} C^{\infty}(X) \stackrel{\pi}{\longrightarrow} \ham(X), 
\ee
%is called the \emph{Kostant-Souriau extension}. 
called the \emph{Kostant-Souriau extension},
plays a central role in geometric quantization \cite{Kostant1970,Souriau1970}.
Evaluation in $x \in X$ is a continuous
linear map
$\mathrm{ev}_{x} \colon C^{\infty}(X) \rightarrow \R$ with $\mathrm{ev}_{x} \circ \iota = \mathrm{Id}_{\R}$, and
the corresponding linear map $\ham(X) \rightarrow C^{\infty}(X)$ maps $X_{f}$ to $f - f(x)$. This 
induces via \eqref{psis} the \emph{Kostant-Souriau cocycle}
\be \label{eq:KScocycle} 
\psi_{\mathrm{KS}}(X_{f}, X_{g}) := \{f,g\}_{x}.
\ee
The class $[\psi_{\mathrm{KS}}] \in H^2(\ham(X),\R)$ is independent of $x\in X$ by 
Prop.~\ref{extensionvscohomology}.

\begin{Remark}[Hamiltonian actions]\label{rk:HamAct}
Let $G$ be a finite dimensional Lie group with Lie algebra $\fg$
and $X$ a connected symplectic manifold.
An action $A \colon G \rightarrow \mathrm{Diff}(X)$ 
of $G$ on $X$
is called 
\emph{symplectic} if it preserves $\omega$, and
\emph{weakly hamiltonian} (cf.\ \cite[\S 5.2]{SalamonMcDuff1995}) 
if the image of the infinitesimal action $a \colon \fg \rightarrow \smp(X)$
lies in $\ham(X) \subseteq \smp(X)$.
A
\emph{comomentum map} is
a linear splitting of \eqref{eq:KSextension} along $a \colon \fg \rightarrow \ham(X)$, that is, a linear map
$J \colon \fg \rightarrow C^{\infty}(X)$
such that
$a(\xi) = X_{J(\xi)}$ for all $\xi \in \fg$.
It determines a \emph{momentum map} $\mu \colon X \rightarrow \fg^*$ by $\mu(x)(\xi) = J(\xi)(x)$.
We then have the following commutative diagram:
%\begin{center}
%\centerline{
\begin{equation}\label{eq:cdmomentum}
\xymatrixcolsep{1.5pc}
\xymatrix{
\R \ar@{^{(}->}[r]& \ar[r] C^{\infty}(X)\ar@{->>}[r]&\ham(X)\ar@{^{(}->}[r]&\smp(X)\ar@{->>}[r] & H^1_{\mathrm{dR}}(X)\\
&&& \fg\ar[u]_{a}\ar[ur]_{0} \ar[ul]_{a}\ar[ull]^{J} &\,.
}
\end{equation}
The action is called \emph{hamiltonian} if there exists a comomentum map $J \colon \fg \rightarrow C^{\infty}(X)$
which is a Lie algebra homomorphism. 
The failure of $J$ to be a homomorphism is measured by the 2-cocycle
$\psi_{J} \colon \fg\times \fg \rightarrow \R$, 
defined (cf.\ \cite[p.\ 109]{Souriau1970}) by 
\begin{equation}\label{eq:noneq2c}
\psi_{J}(\xi, \eta) = \{J(\xi), J(\eta)\} - J([\xi,\eta])\,.
\end{equation} 
Since $X$ is connected, the difference $c = J' - J$ between two comomentum maps is a 
1-cochain $c\colon \fg \rightarrow \R$, and $\psi_{J'} - \psi_{J} = \delta c$.
It follows that the class $[\psi_{J}] \in H^2(\fg,\R)$ does not depend on the choice of $J$,
and is zero if and only if the action is hamiltonian.
The Kostant-Souriau class 
$[\psi_{KS}]$ is \emph{universal} for these non-equivariance classes, in the sense 
that 
$[\psi_{J}] \in H^2(\fg,\R)$ is the pullback of
$[\psi_{KS}] \in H^2(\ham(X),\R)$ along the infinitesimal action $a\colon \fg \rightarrow \ham(X)$.
\end{Remark}

The following corollary shows that $[\psi_{\mathrm{KS}}] = 0$ if and only if $X$ is compact.
In particular,
%By Remark \eqref{rk:HamAct}, this implies in 
%particular that 
every weakly hamiltonian action on a compact, connected, 
symplectic manifold is hamiltonian.

\begin{Corollary}\label{corr:KSextension}
The Kostant-Souriau extension
is split if and only if $X$ is compact.
The splitting $\langle \lambda_{X}\rangle \colon C^{\infty}(X) \rightarrow \R$ 
is then given by 
\be\label{eq:normalisedcharacter}
\langle \lambda_{X}\rangle(f) = \frac{1}{\mathrm{vol}(X)} \int_{X} f \vol,
\ee
with $\mathrm{vol}(X) := \int_{X} \vol$ the symplectic volume of $X$.
\end{Corollary}
\begin{proof}
If $X$ is compact, then $\langle \lambda_{X}\rangle$ is  a continuous splitting; it is a multiple
of the continuous character $\lambda_{X}$ of eqn.~(\ref{eencocykel}), and it
manifestly satisfies $\langle \lambda_{X}\rangle \circ \iota = \mathrm{Id}_{\R}$.
If $X$ is noncompact, there are no nontrivial characters, continuous or not,
since $C^\oo(X)$ is perfect by Prop.~\ref{Prop:noncpspoisson},
so the extension~(\ref{eq:KSextension}) is not split.
\end{proof}

If $X$ is compact connected, we thus have an isomorphism
of $\ham(X)$ with $\mathrm{Ker}(\langle \lambda_{X}\rangle) =C^\oo_{c,0}(X)$, and
$C^{\infty}(X) \simeq C^\oo_{c,0}(X) \oplus \R$.

\begin{Corollary}\label{hamp} {\rm (\cite[\S 12]{Lichnerowicz74})} Let $X$ be a symplectic manifold. Then
the Lie algebra $\ham(X)$ of hamiltonian vector fields is perfect. In particular, the first Lie 
algebra cohomology $H^1(\ham(X),\R)$ is trivial.
\end{Corollary}
\begin{proof}
If a connected component $X_{x}$ of $X$ is compact, 
then $\ham(X_{x}) \simeq C^\oo_{c,0}(X)$ is perfect by 
Prop.~\ref{ALDM}. If $X_{x}$ is noncompact, then $\ham(X_{x})$
is perfect as the homomorphic image of the perfect (Prop.~\ref{Prop:noncpspoisson}) 
Lie algebra $C^{\infty}(X_{x})$. 
When restricted to $X_{x}$, the hamiltonian vector field $X_{f}$ can be written as a sum of at most 
$2n(2n+1)$ commutators $[X^{x}_{G^{i}_{r}}, X^{x}_{H^{i}_{r}}]$, cf.\ Rk.~\ref{Rk:numbercomm}.
Since $X^{x}_{G^{i}_{r}}$ and $X^{x'}_{H^{i}_{r}}$ commute if they live
on different connected components,
we can write $X_f$ as the sum of at most 
$2n(2n+1)$ commutators $\sum_{X_{x}} X^{x}_{G^{i}_{r}}$ and 
$\sum_{X_{x}} X^{x}_{H^{i}_{r}}$.
\end{proof}

\section{Continuous central extensions}\label{s4}

Having determined the characters of the Lie algebras at hand, we now turn to the 
heart of the paper: classifying the continuous central extensions of 
the compactly supported Poisson algebra $C^{\infty}_{c}(X)$.
From this, we derive a similar classification for $C^{\infty}(X)$ and $\ham(X)$.

\subsection{Diagonal cocycles}

The first step is to show that the Lie algebras of interest to us have \emph{diagonal} cocycles.
Lie algebra cohomology with diagonal cocycles is extensively developed in the monograph 
\cite{Fu}.

\begin{Definition}
A 2-cocycle $\psi$ on a Lie subalgebra $\fg$ of $C^{\infty}(X)$ or $\mathrm{vec}(X)$ 
is called \emph{diagonal} if $\psi(f,g) = 0$ whenever
$\mathrm{supp}(f) \cap \mathrm{supp}(g) = \emptyset$.
\end{Definition}
The easiest case is the perfect Lie algebra $\fg = C^\oo_{c,0}(X)$. 
\begin{Proposition}\label{goonaal}
Let $X$ be a symplectic manifold. Then
every 2-cocycle on $C^\oo_{c,0}(X)$, continuous or not, is diagonal.
%If $X$ is connected and noncompact, the same holds for 
%$C^{\infty}(X)$.
\end{Proposition}
\begin{proof}
Let $f, g \in C^\oo_{c,0}(X)$ 
be such that $f$ and $g$ have disjoint support. Because $X$ is a 
normal space,
%(remember that $X$ is assumed connected), 
one can choose disjoint open sets $U,V \subset X$ with
$\mathrm{supp}(f) \subset U$ and $\mathrm{supp}(g) \subset V$.
According to Lemma \ref{ALDM} applied to $C^\oo_{c,0}(V)$, one can write $g$ as a sum of commutators
$g = \sum_{i=1}^{n}\{\phi_{i},\xi_{i}\}$ where $\phi_{i}$ and $\xi_{i}$
are contained in $C^\oo_{c,0}(V)$ and whose support is therfore disjoint
with that of $df$. Using the cocycle identity for $\psi$
and the fact that $f$ commutes with $\phi_{i}$ and $\xi_{i}$, we find
\begin{equation}\label{trick}
\psi(f,g) = \sum_{i=1}^{n}\psi(f,\{\phi_{i},\xi_{i}\}) =
- \sum_{i=1}^{n} \psi(\phi_{i},\{\xi_{i}, f\}) +
\psi(\xi_{i},\{f, \phi_{i}\}) = 0\,.
\end{equation}
This proves that every cocycle on $C^\oo_{c,0}(X)$ is diagonal.
\end{proof}

Building on the above trick, we derive the corresponding result for the Poisson Lie algebra.
\begin{Proposition}\label{ookgoonaal}
For a connected symplectic manifold $X$, 
every 2-cocycle on $C_{c}^{\infty}(X)$ and $C^{\infty}(X)$, continuous or not, is diagonal. 
\end{Proposition}
\begin{proof}
We start with $C^{\infty}(X)$ for $X$ noncompact.
For $f, g \in C^{\infty}(X)$ with disjoint support,
we find disjoint open sets 
$U, V \subseteq X$ with $\mathrm{supp}(f) \subset U$ and $\mathrm{supp}(g) \subset V$.
Since $V$ has no compact connected components,
we can apply Lemma~\ref{Prop:noncpspoisson}
to $C^{\infty}(V)$ to write $g$ as a finite sum of commutators with support in $V$, and 
follow the proof of Prop.~\ref{goonaal} to
conclude that every 2-cocycle on $C^{\infty}(X)$ is diagonal.

Next we pass to $C^\oo_c(X)$ for $X$ noncompact.
Let $\psi$ be a 2-cocycle on $C_{c}^{\infty}(X)$. For $f\in C_{c}^{\infty}(X)$
and $g\in C^\oo_{c,0}(X)$ with disjoint support, 
the proof of 
equation (\ref{trick}) can be followed word by word to show that 
$\psi(f,g) = 0$. For $f, g \in C^{\infty}_{c}(X)$,
we now define \[\tilde{\psi}(f,g) := \psi(f, g - \Delta_g)\,,\]
where $\Delta_g \in C_{c}^{\infty}(X)$ satisfies 
$\int_{X} \Delta_g \omega^{n} = \int_{X} g \omega^{n}$
(so that $g - \Delta_g \in C^\oo_{c,0}(X)$) and 
$\mathrm{supp}(f) \cap \mathrm{supp}(\Delta_g) = \emptyset$.
This does not depend on the choice of $\Delta_g$; 
another choice $\Delta'_g$ yields
$\Delta'_g - \Delta_g \in C^\oo_{c,0}(X)$, and since
$\mathrm{supp}(f)$ is disjoint from
$\mathrm{supp}(\Delta'_g- \Delta_g)$, 
we have 
$\psi(f,\Delta'_g - \Delta_g) = 0$. Note that if $f$ and $g$ have disjoint 
support, then we may choose $g = \Delta_g$, so that $\widetilde{\psi}(f,g) = 0$,
i.e. $\widetilde\ps$ is diagonal.

We show that $\widetilde{\psi}$ is equal to $\psi$.
The bilinear map 
\[\psi - \widetilde{\psi} \colon C^{\infty}_{c}(X) \times C^{\infty}_{c}(X) \rightarrow \R\]
vanishes on $C^{\infty}_{c}(X) \times C^\oo_{c,0}(X)$ by definition, and it vanishes on 
$C^\oo_{c,0}(X) \times C^{\infty}_{c}(X)$ because for $f_0 \in C^\oo_{c,0}(X)$ and $g \in C^{\infty}_{c}(X)$,
disjointness of $\mathrm{supp}(f_0)$ and $\mathrm{supp}(\Delta_g)$ implies $\psi(f_0, \Delta_g) = 0$,
hence $\widetilde{\psi}(f_0,g) = \psi(f_0, g)$. 
It follows that $\psi - \widetilde{\psi}$ factors through a bilinear map on 
%the 1-dimensional vector space
$C^{\infty}_{c}(X)/C^\oo_{c,0}(X)$. This map is antisymmetric because 
any two classes $[f], [g] \in C^{\infty}_{c}(X)/C^\oo_{c,0}(X)$ have representatives 
with disjoint support, so that $\widetilde{\psi}(f,g) = 0$ 
and $(\psi - \widetilde{\psi})([f], [g]) = \psi(f,g) = -\psi(g,f) = -(\psi - \widetilde{\psi})([g],[f])$.
Since $X$ is connected, $C^{\infty}_{c}(X)/C^\oo_{c,0}(X)$ is 1-dimensional, so
$\psi - \widetilde{\psi}$ is zero. It follows that $\psi = \widetilde{\psi}$ is diagonal.

It remains to show that for $X$ compact connected, every 2-cocycle  
$\psi$ on $C^{\infty}(X)$ is diagonal.
Let $f$ and $g$ have disjoint nonempty support. Since $X$
is connected, the complement of the union of the closed 
sets $\mathrm{supp}(f)$ and $\mathrm{supp}(g)$ cannot be empty,
and contains some point $x_0$.
Now $\psi$ restricts to a cocycle $\psi_{X - \{x_0\}}$ on $C_{c}^{\infty}(X - \{x_0\})$,
which is diagonal because $X - \{x_0\}$ is noncompact. 
In particular, $\psi(f,g) = \psi_{X - \{x_0\}}(f,g) = 0$, so that $\psi$ is diagonal. 
\end{proof}

In the remainder of this section, we will focus mainly on the second Lie algebra cohomology 
of the Poisson Lie algebra. 
The following proposition 
shows that this suffices in order to determine $H^2(\ham(X),\R)$.

\begin{Proposition} \label{kaf}
For any connected symplectic manifold $X$,
the canonical projection 
%$C^{\infty}(X)\rightarrow \ham(X) \colon 
$f \mapsto X_{f}$
induces a surjection $H^2(\ham(X),\R) \rightarrow H^2(C^{\infty}(X),\R)$.
If $X$ is compact, this is an isomorphism 
\[H^2(\ham(X),\R) \simeq H^2(C^{\infty}(X),\R)\,.\]
If $X$ is noncompact, then its kernel is 1-dimensional, 
spanned by the Kostant-Souriau class
$[\psi_{KS}]$ of equation~(\ref{eq:KScocycle}). 
Every splitting of the exact sequence 
\[
\R[\psi_{KS}] \rightarrow H^2(\ham(X),R) \rightarrow H^2(C^{\infty}(X),\R)
\]
of vector spaces
yields an isomorphism 
\[
H^2(\ham(X),R) \simeq H^2(C^{\infty}(X),\R) \oplus \R [\psi_{KS}].
\]
%
%Let $X$ be a connected symplectic manifold.
%If $X$ is compact, then
%$H^2(\ham(X),\R)$ is isomorphic to $H^2(C^{\infty}(X),\R)$.
%If $X$ is noncompact, then 
%it is isomorphic to $H^2(C^{\infty}(X),\R) \oplus \R [\psi_{KS}]$,
%where $\psi_{KS}$ is the Kostant-Souriau 
%cocycle of equation~(\ref{eq:KScocycle}).
\end{Proposition}
\begin{proof}
Every cocycle $\psi \colon C^{\infty}(X) \times C^{\infty}(X) \rightarrow \R$
vanishes on $C^{\infty}(X) \times \R \one$.
Indeed, write $f = \sum_{i=1}^{m}\{g_i,h_i\} + c\one$. (We have
$C^{\infty}(X) = [C^{\infty}(X),C^{\infty}(X)]$ if $X$ is noncompact and
$C^{\infty}(X) = [C^{\infty}(X),C^{\infty}(X)] \oplus \R \one$ if $X$ is compact.)
Then 
\be\label{nulopeen}
\psi(f,\one) = \sum_{i=1}^{m}\psi(\{g_i,h_i\},\one) = -\psi(\{h_i,\one\},g_i) 
-\psi(\{\one, g_i\},h_i) = 0\,.\ee
Since $\ham(X) = C^{\infty}(X)/\R\one$, 2-cocycles on $\ham(X)$ correspond precisely to
2-cocycles on $C^{\infty}(X)$. 
A 1-cochain $\chi$ on $C^{\infty}(X)$ corresponds to a 1-cochain on $\ham(X)$
if and only if $\chi(\one) = 0$. 
Writing $\chi = \chi_0 + \chi(\one)\mathrm{ev}_{x}$ with 
$\chi_0 := \chi - \chi(\one) \mathrm{ev}_{x}$ 
a 1-cochain on $C^{\infty}(X)/\R\one \simeq \ham(X)$, 
we see that 
every 2-coboundary $\dd \chi$ on $C^{\infty}(X)$
is the sum of a coboundary $\dd \chi_0$ in $\ham(X)$ and
a multiple of the Kostant-Souriau cocycle $\dd\mathrm{ev}_{x} = \psi_{\mathrm{KS}}$.
The latter is cohomologous to zero in $\ham(X)$ if and only if $X$
is compact by Corrolary~\ref{corr:KSextension}, so the result follows. 
\end{proof}

\subsection{Cocycles and coadjoint derivations}

If $\fg$ is a locally convex topological Lie algebra,
we denote by $\mathfrak{g}'$ the continuous coadjoint representation of $\mathfrak{g}$,
that is, the continuous dual with the action 
\[(\mathrm{ad}^{*}_{f} \phi)(g):= \phi (-[f,g])\,.\]
A \emph{(coadjoint) derivation} $D \colon \mathfrak{g} \rightarrow \mathfrak{g}'$
is a linear map satisfying 
\[D([f,g]) = \mathrm{ad}^{*}_{f}D(g) - \mathrm{ad}^{*}_{g}D(f)\,.\]
A derivation $\mathfrak{g} \rightarrow \mathfrak{g}' $ is \emph{skew symmetric} 
if $D(f)(g) + D(g)(f) = 0$ for all $f,g\in \mathfrak{g}$, and \emph{inner} if it is of the form
$D(f) = \mathrm{ad}^{*}_{f} H$ for some $H \in \mathfrak{g}'$.
It will be convenient to formulate the second Lie algebra cohomology in terms 
of skew symmetric derivations.

\begin{Proposition}\label{derivcoc}
If $\fg$ is a locally convex Lie algebra, then every continuous 2-cocycle 
$\psi \colon \fg \times \fg \rightarrow \R$ induces a skew symmetric derivation
\[D_{\psi}(f) \colon g \mapsto \psi(f,g)\,.\] 
It is a coboundary  (that is, $\psi = \dd H$ for some $H \in \fg'$) 
if and only if $D_{\psi}$ is inner (that is, $D_{\psi}(f) = \mathrm{ad}^{*}_{f} H$ for some $H \in \fg'$).
Conversely, every skew symmetric derivation $D$ induces a 
separately continuous 2-cocycle
\begin{equation}\label{psiD}
\psi_{D}(f,g) := D(f)(g)
\end{equation}
on $\fg$. 
If $\fg$ is either a Fr\'echet Lie algebra
or if $\fg = C^{\infty}_{c}(X)$, then $\psi_{D}$ is automatically jointly continuous.
In this case, then, 
$H^2(\mathfrak{g},\R)$ is the space of skew symmetric derivations
modulo the inner ones. 
\end{Proposition}
\begin{proof}
Since $\psi$ is continuous, so is $D_{\psi}(f) \colon \fg \rightarrow \R$. 
Skew symmetry follows from $\psi(f,g) = -\psi(g,f)$, and the cocycle condition 
\[\psi(\{f,g\},h) = \psi(g,-\{f,h\}) - \psi(f,-\{g,h\})\] is precisely the requirement that 
$D_{\psi}$ be a derivation. Conversely, every skew symmetric derivation $D\colon \mathfrak{g} \rightarrow \mathfrak{g}'$ 
defines a 2-cocycle $\psi_{D}$ by \eqref{psiD}.
The derivation $D_{\psi}$ gives back the cocycle $\psi$, 
and $\psi = \dd H$ if and only if $D_{\psi}$ is inner, $D_{\psi}(f) = \mathrm{ad}^{*}_{f}H$.
Since $\psi_{D}$ is skew symmetric 
and continuous in the right argument, it is a \emph{separately} continuous bilinear map on $\fg$,
and every inner derivation defines a jointly continuous coboundary.

If $\fg$ is a Fr\'echet Lie algebra (such as $\fg = C^{\infty}(X)$ or $\fg = \ham(X)$), 
then $\psi_{D}$ is jointly continuous 
by \cite[Thm.~2.17]{Rudin1991}.
If $\fg = C^{\infty}_{c}(X)$ for a symplectic manifold $X$, then by skew symmetry,
$D \colon C^{\infty}_{c}(X) \rightarrow C^{\infty}_{c}(X)'$
is continuous for the weak topology, hence given by a distribution on $X \times X$
by the Schwartz Kernel Theorem \cite[\S 23.9.2]{Dieudonne7}.
In particular, it is jointly continuous.
\end{proof}

\begin{Remark}\label{Rk:mapH1H2}
The first Lie algebra cohomology $H^1(\fg, \fg')$ of $\fg$ with values in $\fg'$ 
is defined as the vector space of derivations $D \colon \fg \rightarrow \fg'$ 
modulo the inner derivations.  
Note that a 2-cocycle $\psi$ is a coboundary,
 $\psi = \dd H$ for some $H\in\g'$, if and only if $D_{\psi}$ is inner, $D_{\psi}(f) = \mathrm{ad}^{*}_{f}H$.
From Proposition \ref{derivcoc}, we thus obtain a natural map
\begin{equation}\label{eq:mapH1H2}
H^2(\g,\R)\to H^1(\g,\g') \,\colon \quad [\psi]\mapsto[D_\psi]\,.
\end{equation}
\end{Remark}

\begin{Remark}
For a weakly hamiltonian action $A \colon G \rightarrow \mathrm{Diff}(X)$
(cf.\ Remark \ref{rk:HamAct}), the infinitesimal action $a \colon \fg \rightarrow \ham(X)$
is $G$-equivariant. Since also the projection $\pi \colon C^{\infty}(X) \rightarrow \ham(X)$
is $G$-equivariant, commutativity of the diagram
\[
\xymatrixcolsep{1.5pc}
\xymatrixrowsep{1.5pc}
\xymatrix{
\R \ar@{^{(}->}[r]& C^{\infty}(X)\ar@{->>}[r]^{\pi}&\ham(X)\\
&& \fg\ar[u]_{a}\ar[ul]^{J} &
}
\]
shows that $\theta_{J}(g) := A(g)^*\circ J - J \circ \mathrm{Ad}_{g^{-1}}$ 
satisfies $\pi \circ \theta_{J}(g) = 0$, hence takes values in $\R$. 
The map $\theta_{J} \colon G \rightarrow \fg'$ 
is a 1-cocycle, $\theta_{J}(gh) = \theta_{J}(g) + \mathrm{Ad}^{*}_{g}\theta_{J}(h)$, 
whose class $[\theta_{J}] \in H^1(G,\fg')$ does not depend on the choice of $J$.
It vanishes if and only if the action has a $G$-equivariant comomentum map
\cite[{p.~108--115}]{Souriau1970}.
The derived class $[D\theta_{J}] \in H^1(\fg,\fg')$ in Lie algebra cohomology 
is the image of the class 
$[\psi_{J}] \in H^2(\fg, \R)$ of Remark \ref{rk:HamAct} under 
the map $H^2(\g,\R)\to H^1(\g,\g')$ of %equation \eqref{eq:mapH1H2}.
Remark \ref{Rk:mapH1H2}.
\end{Remark}

In the following lemma, we use Peetre's Theorem to show that
for $\fg = C^{\infty}_{c}(X)$, every skew-symmetric derivation is a differential operator.
We choose a locally finite cover of $X$
by open, relatively compact neighbourhoods $U_{i}$
with Darboux coordinates 
$x^{\mu}$ with $\mu \in \{1, \ldots, 2n\}$.
Using the summation convention, the Poisson bracket in Darboux coordinates
becomes 
\[
\{f,g\}=\om^{\si\ta}\partial_\si f\partial_\ta g\,,
\]
where $\om_{\si\ta}$ denote the elements of the matrix 
\[J=\left(\begin{array}{ccc} 0 & I \\ -I & 0 \end{array}\right)\] and $\om^{\si\ta}=-\om_{\si\ta}$
denote the elements of its inverse $J^{-1}=-J$.

We write $\partial_{\vec{\mu}}$ for 
$\partial_{\mu_1} \ldots \partial_{\mu_{N}}$ if $\vec{\mu}$ is the multi-index
$\vec{\mu} = (\mu_1, \ldots, \mu_{N})$ in $\{1, \ldots, 2n\}^{N}$.
Note that if $\vec{\mu}'$ is a permutation of $\vec{\mu}$, i.e.\
$\mu'_{i} = \mu_{\sigma(i)}$ with $\sigma \in S_{N}$,
then $\partial_{\vec{\mu}'} = \partial_{\vec{\mu}}$. 
We write $\vec{\mu}\vec{\nu}$ for the concatenation in 
$\{1, \ldots, 2n\}^{N+M}$ of $\vec{\mu} \in \{1, \ldots, 2n\}^{N}$ and 
$\vec{\nu} \in \{1, \ldots, 2n\}^{M}$, and we write $|\vec{\mu}|$
for the length $N$ of $\vec{\mu} \in \{1, \ldots, 2n\}^{N}$.
There is one multi-index $\vec{\mu} = *$ for $N=0$. 

\begin{Lemma} \label{Peetre}
Every skew symmetric derivation $D \colon C^{\infty}_{c}(X) \rightarrow C^{\infty}_{c}(X)'$
is support decreasing, hence determined by its restrictions 
$D_{U_i} \colon C^{\infty}_{c}(U_i) \rightarrow C^{\infty}_{c}(U_i)'$.
In local coordinates, it is
a differential operator 
\be\label{localder}
D_{U_i}(f) = \sum_{\vec{\mu}} (\partial_{\vec{\mu}} f) S^{\vec{\mu}}
\ee
of locally finite order, with distributions $S^{\vec{\mu}} \in C^{\infty}_{c}(U_i)'$
that are uniquely determined  by $D$ and the requirement that they are invariant 
under permutations of~$\vec{\mu}$.
\end{Lemma}
\begin{proof}
Since $\psi_{D}$ is diagonal (Prop.~\ref{ookgoonaal}), 
$D$ is support decreasing, $\mathrm{supp}(D(f)) \subseteq \mathrm{supp}(f)$.
The fact that $\psi_{D}$ is jointly continuous (Prop.~\ref{derivcoc})
implies that the derivation $D \colon C^{\infty}_{c}(X) \rightarrow C^{\infty}_{c}(X)'$ is 
continuous in the topology on $C^{\infty}_{c}(X)'$ obtained from the seminorms $p^*_{K,\vec{\mu}}$
dual to the seminorms $p_{K,\vec{\mu}}$
that define the Fr\'echet topology on $C^{\infty}_{K}(X)$. 
If we choose a partition of unity $\lambda_i$ subordinate to the cover $U_i$,
then the expression for $D_{U_i}$ follows from
Peetre's theorem \cite[Thm.~1]{Pe} 
applied to the support decreasing maps 
$\lambda_j D(\lambda_i \,\cdot\,) \colon C^{\infty}_{c}(U_{i}) \rightarrow C^{\infty}_{c}(U_j)'$,
which have  empty set of discontinuities. 
\end{proof}

If we consider $C^{\infty}(X)$ instead of $C^{\infty}_{c}(X)$, then the only thing that changes is that
the relevant cochains must have compact support. The following lemma gives the relation
between the two situations, and essentially allows us to restrict attention to $C^{\infty}_{c}(X)$.

 \begin{Lemma}\label{supportcompact}
The cocycle $\psi_{D}$ extends continuously to $C^{\infty}(X)$ if and only if 
$D$ is of 
compact support. If $\psi_{D} = \dd \chi$ for 
$\chi \in C^{\infty}_{c}(X)'$, then $\chi$ extends to $C^{\infty}(X)$.
In particular, the map $\iota^* \colon H^2(C^{\infty}(X),\R) \rightarrow H^2(C^{\infty}_{c}(X),\R)$
induced by
the inclusion $\iota \colon C_{c}^{\infty}(X) \rightarrow C^{\infty}(X)$,
is injective.
\end{Lemma}
\begin{proof}
Suppose that
$\psi_{D}$ extends to a continuous cocycle on $C^{\infty}(X)$.
If $\mathrm{supp}(D)$ were not compact, there would 
exist a countably infinite, locally finite set of points 
$x_i \in \mathrm{supp}(\psi_{D})$. 
We would find disjoint open sets  $V_i \ni x_i$ 
and functions $f_i, g_i$, supported in $V_i$, 
such that $D(f_i)(g_i) = 1$. 
Since the sum $\sum_{i=1}^{\infty} (f_i, g_i)$ converges in 
$C^{\infty}(X)\times C^{\infty}(X)$, so would 
$\sum_{i=1}^{\infty} \psi_{D}(f_i, g_i) =  \sum_{i=1}^{\infty} 1$,
which is absurd. We conclude that if $\psi_{D}$ extends continuously 
to $C^{\infty}(X)$, then $\mathrm{supp}(\psi_{D})$ is compact. 
The converse implication is clear.

It remains to show that if $\psi_{D}$ extends
to $C^{\infty}(X)$ and $\psi_D = \dd \chi$ on $C^{\infty}_{c}(X)$, then 
$\mathrm{supp}(\chi)$ is compact, so that $\chi$ extends to $C^{\infty}(X)$. 
Suppose $\mathrm{supp}(\chi)$ were not compact. Then again, we find 
disjoint open sets $V_i \subseteq X$ and $F_i \in C^{\infty}_{c}(V_i)$
with $\chi(F_i) = 1$. 
Using Remark~\ref{Rk:numbercomm} to write 
$F_i = - \sum_{k=1}^{2n(2n+1)} \{f_i^{k}, g_i^{k}\}$
with $f_i^k, g_i^k \in C^{\infty}_{c}(V_i)$, we find
$
\sum_{k=1}^{2n(2n+1)} \psi(f^k_i, g_i^k) = 1
$.
Since $\sum_{i=1}^{\infty} (f^k_i, g^k_i)$ converges in $C^{\infty}(X)\x C^\oo(X)$,
we see that $\sum_{i=1}^{\infty} \sum_{k=1}^{2n(2n+1)} \psi(f^k_i, g_i^k) = \sum_{i=1}^{\infty}1$ converges, 
which is absurd. We conclude that
$\mathrm{supp}(\chi)$ is compact.
\end{proof}

Sanctioned by Lemma~\ref{supportcompact}, we restrict attention to $C^{\infty}_{c}(X)$.
By looking at the cocycle equation in local coordinates, we now prove that the differential
operator $D$ corresponding to a continuous 2-cocycle $\psi_{D}$ can have degree 
at most~1. In Lemma~\ref{derivationsII}, we will use this to derive explicit local 
expressions for the cocycles.

\begin{Lemma}\label{derivations}
If a differential operator 
$D\colon C^{\infty}_{c}(X) \rightarrow C^{\infty}_{c}(X)'$
is a derivation, then it is of order at most 1.
\end{Lemma}
\begin{proof}
We express the condition 
\be\label{nocoorder}
D(\{f,g\}) - \mathrm{ad}^{*}_{f}D(g) + \mathrm{ad}^{*}_{g}D(f) = 0
\ee
that $D$ be a derivation
in local Darboux coordinates.
Using the summation convention and writing $f_{\vec{\mu}}$ for $\partial_{\vec{\mu}}f$,
we have $D_{U_i}(f) = f_{\vec{\mu}}S^{\vec{\mu}}$.
The formula 
\[\mathrm{ad}^{*}_{f} \phi = \omega^{\sigma\tau}f_{\sigma} \phi_{\tau}\]
follows from 
$\mathrm{ad}^{*}_{f} \phi (g) = \phi(-\omega^{\sigma\tau} 
f_{\sigma}g_{\tau} ) = 
\omega^{\sigma\tau}\partial_{\tau} (f_{\sigma} \phi)(g)$
and 
the antisymmetry of $\omega$, which implies $\omega^{\sigma\tau}(f_{\sigma\tau}) = 0$.
Writing $S^{\vec{\mu}}_{,\tau}$ for
$\partial_{\tau} S^{\vec{\mu}}$, 
and using that $\omega^{\sigma\tau}$ is constant in Darboux coordinates,
we find the
local expression for equation~(\ref{nocoorder})
\be \label{coorder}
%\textstyle
\omega^{\sigma\tau} \Big(
\big(\sum_{\vec{\alpha} \cup \vec{\beta} = \vec{\mu}} 
%{\textstyle \binom{\vec{\mu}}{\vec{\alpha}}}
f_{\sigma \vec{\alpha}}g_{\tau\vec{\beta}} \big)
S^{\vec{\mu}}
- f_{\sigma} g_{\vec{\mu}} S^{\vec{\mu}}_{,\tau}
-  f_{\sigma} g_{\tau \vec{\mu}} S^{\vec{\mu}}
+  g_{\sigma} f_{\vec{\mu}} S^{\vec{\mu}}_{,\tau}
+  g_{\sigma} f_{\tau \vec{\mu}} S^{\vec{\mu}} \Big) = 0\,,
\ee
where 
$\sum_{\vec{\alpha} \cup \vec{\beta} = \vec{\mu}}$ denotes the sum over the $2^{N}$
decompositions of $\vec{\mu} = (\mu_{1}, \ldots \mu_{N})$
into $\vec{\alpha} = (\mu_{i_1}, \ldots, \mu_{i_{K}})$
and $\vec{\beta} = (\mu_{j_{1}}, \ldots, \mu_{j_{N-K}})$,
where $1 \leq i_1 < \ldots, i_{K}\leq N$ and
$1 \leq j_1 < \ldots < j_{N-K} \leq N$ complement 
each other to $1 < \ldots < N$. 
We can write (\ref{coorder}) as
\be\label{defq}
f_{\vec{\mu}}g_{\vec{\nu}} Q^{\vec{\mu};\vec{\nu}} = 0
\ee
for the distributions
\be \label{defQuitgeschreven}
Q^{\vec{\mu};\vec{\nu}} := 
{\textstyle \frac{(|\vec{\mu}| + |\vec{\nu}|-2)!}{|\vec{\mu}|!|\vec{\nu}|!}}
\sum_{\sigma \cup \vec{\alpha} = \vec{\mu}} 
\sum_{\tau \cup \vec{\beta} = \vec{\nu}} 
\omega^{\sigma\tau} S^{\vec{\alpha}\vec{\beta}}
+ \delta_{|\vec{\nu}|, 1} Y^{\nu;  \vec{\mu}}
- \delta_{|\vec{\mu}|, 1} Y^{\mu;  \vec{\nu}}
\ee
where the first term on the r.h.s.\ is zero for $|\vec{\mu}| = 0$ or $|\vec{\nu}| = 0$, and where
\be\label{defY} 
Y^{\nu;\vec{\mu}}
:=
\omega^{\nu\tau}S^{\vec{\mu}}_{,\tau} + 
{\textstyle \frac{1}{|\vec{\alpha}|+1}}
\sum_{\tau \cup \vec{\alpha} = \vec{\mu}} \omega^{\nu\tau}S^{\vec{\alpha}}\,.
\ee
Since the distributions $Q^{\vec{\mu};\vec{\nu}}$
are manifestly independent of $f$ and $g$, and 
symmetric under permutation of
the $\vec{\mu}$ and $\vec{\nu}$ indices separately, the fact that
(\ref{defq}) holds for arbitrary
$f$ and $g$ implies 
\be\label{vergelijkingdienulis}
Q^{\vec{\mu};\vec{\nu}} = 0\,.
\ee
%and symmetric under permutation of
%the $\vec{\mu}$ and $\vec{\nu}$ indices separately.
%Because (\ref{defq}) holds for arbitrary $f$ and $g$, 
%$Q^{\vec{\mu}\vec{\nu}}$ must vanish identically.
%%%%%%%%%%%%%%%%%%%%%%%%%%%%
%
In order to show that $D$ is a differential operator of degree at most $1$,
we will show that, for any multi-index $\vec{\gamma}$ of length $\ge 2$,
the distribution $S^{\vec{\gamma}}$ is identically zero.

There exists a non-trivial decomposition $\vec{\alpha}\cup\vec{\beta}=\vec{\gamma}$
with $|\vec{\alpha}|=r\ge 1$ and $|\vec{\beta}|=s\ge 1$, since 
$|\vec{\gamma}|\ge 2$.
We know that $Q^{\sigma\cup\vec{\alpha};\tau\cup\vec{\beta}}=0$
for all $\sigma,\tau\in\{1,\dots,2n\}$.
Moreover, for these distributions,
the contribution of $Y$ in the expression \eqref{defQuitgeschreven} is zero,
so, denoting by $K$ the constant $\frac{(r+s)!}{(r+1)!(s+1)!}$,
we have
\begin{equation}\label{defQuitgeschrevenhoog2}
Q^{\sigma\cup\vec{\alpha};\tau\cup\vec{\beta}}=K
\sum_{\sigma' \cup \vec{\alpha}' =\sigma \cup \vec{\alpha}} 
\sum_{\tau' \cup \vec{\beta}' = \tau \cup \vec{\beta}} 
\omega^{\sigma'\tau'} S^{\vec{\alpha}'\vec{\beta}'}.
\end{equation}
We contract $Q^{\sigma\cup\vec{\alpha};\tau\cup\vec{\beta}}$ 
with $\omega_{\tau\sigma} = - \omega^{\tau\sigma}$, the inverse
of $\omega$ satisfying 
$\omega_{\tau\sigma}\omega^{\sigma \gamma} = \delta_{\tau}^{\gamma}$.
The result is 
\begin{equation}\label{K}
\omega_{\tau\sigma}Q^{\sigma\cup\vec{\alpha};\tau\cup\vec{\beta}}
=(2n+r + s)K S^{\vec{\alpha}\vec{\beta}}.
\end{equation}
Indeed, there are 4 types of terms on the r.h.s.\ of (\ref{defQuitgeschrevenhoog2})
when contracted with $\omega_{\sigma\tau}$, 
according to whether $\sigma'$ is equal to $\sigma$ or not, and
whether $\tau'$ is equal to $\tau$ or not.

If $\sigma \neq \sigma'$ and $\tau \neq \tau'$, then
$\sigma$ is in $\vec{\alpha}'$ and $\tau$ is in $\vec{\beta}'$, so the term is
symmetric under $\sigma \leftrightarrow \tau$. It therefore
contracts to zero with the antisymmetric tensor $\omega_{\tau\sigma}$.

If $\sigma = \sigma'$ and $\tau = \tau'$, then because
$\omega_{\tau\sigma}\omega^{\sigma\tau} = 2n$,
contraction yields the term $2nKS^{\vec{\alpha}\vec{\beta}}$.
There is precisely one such term.

If $\sigma = \sigma'$ and $\tau \neq \tau'$, say $\tau' = \beta_{p}$, then 
$\omega_{\tau\sigma}\omega^{\sigma\beta_p} = \delta_{\tau}^{\beta_p}$.
Together with
the symmetry of $S^{\vec{\mu}}$ under permutations of $\vec{\mu}$, this
implies that 
\[\omega_{\tau\sigma}\omega^{\sigma\beta_p}S^{\vec{\alpha}\tau\beta_1 \ldots \hat{\beta}_p \ldots \beta_{s}}
= S^{{\vec{\alpha}\vec{\beta}}}\,.
\]
Since there are $s$ such terms, their contribution is 
$s K S^{{\vec{\alpha}\vec{\beta}}}$.
By a similar reasoning, the terms with $\sigma \neq \sigma'$ and $\tau = \tau'$
have a contribution of
$r K S^{{\vec{\alpha}\vec{\beta}}}$.

The equality \eqref{K}, combined with \eqref{vergelijkingdienulis},
now ensures that $S^{\vec{\gamma}} =S^{\vec{\alpha}\vec{\beta}}=0$
for $\vec{\gamma}$ of length $\geq 2$,
so that $D$ is a differential operator of order at most 1. 
\end{proof}

We continue this line of reasoning to obtain the following more explicit expression for these
derivations. 

\begin{Lemma}\label{derivationsII}
If a differential operator
$D\colon C^{\infty}_{c}(X) \rightarrow C^{\infty}_{c}(X)'$ 
is a derivation, then in local Darboux coordinates,
it is given by a first order differential operator
\[ D_{U_i}(f) = S^{\mu} \partial_{\mu}f +c f\]
for $c\in\RR$ and distributions $S^{\mu}$ with 
\be\label{tw}
\omega^{\nu\tau}\partial_{\tau}S^{\mu} - 
\omega^{\mu\tau} \partial_{\tau} S^{\nu} = c \omega^{\mu\nu}\,.
\ee
The derivation $D$ is antisymmetric if and only if $c=0$.
\end{Lemma}
\begin{Remark}
Contracting both sides of \eqref{tw} with $\omega_{\mu\nu}$,
one obtains \[\partial_{\mu} S^{\mu} = -nc\,.\]
In particular, $\partial_{\mu} S^{\mu} = 0$ if $D$ is antisymmetric.
\end{Remark}

\begin{proof}
The requirement that $D$ be a derivation is equivalent to 
eqn.~(\ref{vergelijkingdienulis}) for the $Q^{\vec{\mu};\vec{\nu}}$
of eqn.~(\ref{defQuitgeschreven}).
Using the antisymmetry of $\omega$, we see that 
\[
Q^{\mu_1\mu_2;\nu_1} = \frac{1}{2}\big(
\omega^{\mu_1\nu_1}S^{\mu_2}
+ \omega^{\mu_2\nu_1}S^{\mu_1}
+ \omega^{\nu_1\mu_1}S^{\mu_2}
+ \omega^{\nu_1\mu_2}S^{\mu_2}
\big)
\]
vanishes identically, as does $Q^{\mu_1;\nu_1\nu_2}$.
%is trivially  
%$Q^{\mu\sigma,\tau} = Q^{\mu, \sigma\tau} = 0$ are trivially satisfied,
For $Q^{\mu;\nu}$ and $Q^{*;\mu} = - Q^{\mu;*}$, eqn.~(\ref{defQuitgeschreven}) 
yields the nontrivial equations
\begin{equation} \label{sympconstraint1}
Q^{\mu;\nu} = \omega^{\nu\tau}\partial_{\tau}S^{\mu} - 
\omega^{\mu\tau} \partial_{\tau} S^{\nu} - \omega^{\mu\nu} S^{*} = 0
\end{equation} 
and
\begin{equation} \label{sympconstraint0}
Q^{*;\mu} =  \omega^{\mu\tau}\partial_{\tau}S^{*} = 0\,.
\end{equation} 
Because $\omega$ is invertible,
the latter equation 
implies that $S^*$ is a constant, ${S^*=c}$. 
(A distribution satisfies $\partial_{\mu} \phi = 0$ for all $\mu$
if and only if it is constant). 
Contracting equation (\ref{sympconstraint1})
with $\omega_{\mu\nu}$ and using antisymmetry of $\omega$,
together with $\omega_{\mu\nu}\omega^{\nu \sigma} = \delta_{\mu}^{\sigma}$
and $\delta_{\mu}^{\mu} = \mathrm{tr}(\one) = 2n$, we obtain
$2\partial_{\mu}S^{\mu} + 2nS^* = 0$, so that
the divergence $ \partial_{\mu}S^{\mu} = -nc$ is constant.
If $D$ is antisymmetric, the equation 
\begin{eqnarray*}
D(f)(g) + D(g)(f) &=& (S^{\mu} (f_{\mu}g) + S^* (fg)) + (S^{\mu} (fg_{\mu}) + S^* (fg))\\
&=& S^{\mu}(\partial_{\mu}(fg)) + 2 S^{*}(fg)\\
&=& (2S^* - \partial_{\mu}S^{\mu})(fg) = 0
\end{eqnarray*}
for all $f,g$ yields $S^* = \frac{1}{2}\partial_{\mu}S^{\mu}$,
and because we already had $S^* = -\frac{1}{n}\partial_{\mu}S^{\mu}$,
we get $S^* = c = 0$.
\end{proof}

\subsection{Second cohomology of the Poisson Lie algebra}

Following Roger \cite[\S 9]{Roger1995}, we now define maps
\be
H^1_{\mathrm{dR}}(X) \rightarrow H^2(C^{\infty}_{c}(X),\R) \label{rogerc}
\ee
and
\be
H^1_{\mathrm{dR},c}(X) \rightarrow H^2(C^{\infty}(X),\R)\,, \label{rogernc}
\ee
%\begin{eqnarray}\label{rogersmapfromto}
%H^1_{\mathrm{dR}}(X) &\rightarrow& H^2(C^{\infty}_{c}(X),\R)\,\quad \mathrm{and}\label{rogerc}\\
%H^1_{\mathrm{dR}, c}(X) &\rightarrow& H^2(C^{\infty}(X),\R)\label{rogernc}
%\end{eqnarray}
from the (compactly supported) de Rham cohomology of $X$ into the
continuous Lie algebra cohomology of
the Poisson Lie algebras $C^{\infty}_{c}(X)$ and $C^{\infty}(X)$.
We then use the preceding results to show that these maps are isomorphisms.
\begin{Proposition}\label{Rogersmap}
If $\alpha \in \Omega^1(X)$ is closed, then 
\be\label{psal} 
\psi_{\alpha}(f,g) := \int_{X} f (i_{X_{g}}\alpha) \,\omega^n/n!
\ee
defines a 2-cocycle on %the Poisson Lie algebra 
$C^{\infty}_{c}(X)$, which is a coboundary if $\alpha$ is exact, $\alpha = d h$ with $h \in \Omega^0(X)$. 
If $\alpha$ is compactly supported, then $\psi_{\alpha}$ extends to a cocycle on 
$C^{\infty}(X)$, which is a coboundary if $\alpha = d h$ with $h\in \Omega^0_{c}(X)$.
In particular, the correspondence
\be \label{rogersmapp}
[\alpha] \mapsto [\psi_{\alpha}]\,.
\ee
yields well defined maps \eqref{rogerc} and \eqref{rogernc}.
\end{Proposition}
\begin{proof}
If $\alpha$ is closed, then the vector field $S$ defined by
$i_{S}\omega = -\alpha$ is symplectic, i.e. $L_{S}\omega = 0$.
In the following, we either have $f,g \in C^{\infty}_{c}(X)$ and $S \in \mathrm{sp}(X)$
or $f,g \in C^{\infty}(X)$ and $S \in \mathrm{sp}_{c}(X)$.
Either way, we have $L_{S}\{f,g\} = \{L_{S}f,g\} + \{f,L_{S}g\}$, so
$L_S$ defines a derivation into $C^{\infty}_{c}(X)$.
If we define a linear functional $\langle \,\cdot \, \rangle$
on the symmetric tensor product $C^{\infty}_{c}(X) \vee C^{\infty}(X)$ by 
\[
%\langle \, \cdot \, \rangle \colon C^{\infty}_{c}(X) \vee C^{\infty}(X) \rightarrow \R
%\quad
\langle f \vee g \rangle := \int_{X} fg \omega^{n}/n!\,,
\]
then $\langle\,\cdot\,\rangle$ is $L_{S}$-invariant in the sense that 
\be \label{Sasym}
\langle L_{S}f \vee g\rangle + \langle f \vee L_{S} g \rangle = 0\,,
\ee
because $\int_{X}(L_{S}f) g \omega^n + \int_{X}f (L_{S}g) \omega^n = \int_{X}L_{S}(fg\omega^n) = 0$. 
Since $i_{X_{f}}\alpha = -L_{S}f$, we can write
\be\label{alternatief}
\psi_{\alpha}(f,g) = - \langle f \vee L_{S}g\rangle\,.
\ee
Equation \eqref{Sasym} then ensures that 
$\psi_{\alpha}$ is skew symmetric, 
$\psi_{\alpha}(f,g) + \psi_{\alpha}(g,f) = 0$.
The cocycle identity follows from
\begin{eqnarray*}
\langle f \vee L_{S}\{g,h\} \rangle + \langle g\vee L_{S}\{h,f\} \rangle + \langle h\vee L_{S}\{f,g\} \rangle &=&\\
\langle f\vee (\{L_{S}g,h\} + \{g, L_{S}h\})  \rangle - \langle L_{S}g\vee \{h,f\} \rangle - \langle L_{S}h\vee \{f,g\} \rangle &=&\\
\langle f\vee (\{L_{S}g,h\} + \{g, L_{S}h\})  \rangle - \langle \{L_{S}g,h\}\vee f \rangle - \langle \{g, L_{S}h\}\vee f \rangle &=& 0\,,
\end{eqnarray*}
where in the last step, we used that  \eqref{Sasym} with $S = X_{h}$ implies 
$\langle \{h,f\} \vee g\rangle + \langle f, \{h,g\}\rangle = 0$.
If $\alpha$ is exact, say $\alpha = dh$ for $h \in C^{\infty}(X)$ 
not necessarily compactly supported, then $S = X_{h}$ is a hamiltonian vector field, and
\be\label{bijdewildespinnen}
\psi_{dh}(f,g) = -\langle f\vee L_{X_{h}} g\rangle = \langle f \vee L_{X_{g}} h \rangle = 
\langle \{f, g\} \vee h\rangle 
\ee
shows that $\psi_{dh} = -\dd \chi_{h}$ for the 1-cochain $\chi_{h}(f) := \langle f \vee h \rangle$.
\end{proof}

We can define a similar map for the Lie algebra $\ham(X)$, but here we have to be slightly careful;
every cocycle $\psi_{\alpha}$ on $C^{\infty}(X)$ yields a cocycle on $\ham(X)$, but
even if $\psi_{\alpha}$ is a coboundary on $C^{\infty}(X)$, it need not be trivial 
on $\ham(X)$. It turns out that if $\alpha = d h$ with $h \in \Omega^0_{c}(X)$, then
$[\psi_{\alpha}] = 0$ in $H^2(C^{\infty}(X),\R)$, but
$[\psi_{\alpha}] = \langle h \rangle [\psi_{KS}]$ in $H^2(\ham(X),\R)$, 
where 
\be \label{gemiddelde}
\langle h \rangle := \lambda(h) = \int_{X} h \omega^n/n!\,.
\ee
%\todo{This is our $\la$ already defined in \eqref{lam}.
%Could we use the same $\la$ here too?
%The same identity comes again in \eqref{average}.}

\begin{Proposition}\label{hamcocycles}
Let $X$ be a symplectic manifold.
Assigning to the closed 1-form $\alpha \in \Omega^1_{c}(X)$ the cocycle 
\be\label{voorham}
\psi_{\alpha}(X_f, X_g) = \int_{X}f\alpha(X_g)\omega^n/n!
\ee
yields a well defined linear map 
$Z^1_{ c}(X)/d\Omega^0_{c,0}(X) \rightarrow H^2(\ham(X),\R)$,
where 
\[
Z^1_{c} (X):= \mathrm{Ker}(d\colon \Omega^1_{c}(X) \rightarrow \Omega^2_{c}(X))
\]
and
\[\Omega^0_{c, 0}(X) := \Big\{h \in \Omega^0_{c}(X)\,;\, \langle h \rangle = 0\Big\}\,.\]
In particular, we have a map $H^1_{\mathrm{dR}}(X) \rightarrow H^2(\ham(X),\R)$ if $X$
is compact.
\end{Proposition}
\begin{proof}
Equation \eqref{alternatief} %, $\psi_{\alpha}(f,g) = - \langle f \vee L_{S}g\rangle$, 
shows that $\psi_{\alpha}(f,g)$
vanishes if $g = 1$, so that $\psi_{\alpha}$ defines a cocycle on $\ham(X)$.
If $\alpha = d h$, then we choose a point $x \in X$, and define the 1-cochain 
$\chi_{h}(X_f) := \int_{X} (f_x - f)h \omega^n/n!$ on $\ham(X)$.
Combining
\[\dd \chi_{h}(X_f,X_g) = \int_{X} \{f,g\} h \omega^n/n! - \{f, g\}_{x}\int_{X} h \omega^n/n! \]
with equations \eqref{bijdewildespinnen}, \eqref{eq:KScocycle}, and \eqref{gemiddelde},
we find $[\psi_{dh}] = \langle h \rangle [\psi_{KS}]$.
\end{proof}

From Prop.~\ref{derivcoc} we know that every
continuous 2-cocycle of the compactly supported Poisson Lie algebra $C^{\infty}_{c}(X)$ comes from a skew symmetric derivation.
These derivations are explicitly described in 
Lemmas \ref{Peetre}, \ref{derivations} and \ref{derivationsII}.
We see that the restriction of the cocycle $\psi$ to a 
Darboux coordinate patch $U_i$ can be written as
\be\label{explicit}
\psi(f,g) = S_{U_i}^{\mu}(g\partial_{\mu}f)
\ee
for distributions $S_{U_i}^{\mu}$ on $U_i$
satisfying 
\be
\label{es}
\omega^{\nu\tau}\partial_{\tau}S_{U_i}^{\mu} - \omega^{\mu\tau}\partial_{\tau}S_{U_i}^{\nu} = 0.
\ee
We now use this explicit description of 2-cocycles 
%obtained in Lemmas \ref{Peetre}, \ref{derivations} and \ref{derivationsII}
to show that the map (\ref{rogerc}) is an isomorphism.

\begin{Theorem}\label{HoofdZaak}
For any symplectic manifold $X$, the map $[\al]\mapsto [\ps_\al]$ described in Prop.~\ref{Rogersmap}
is an isomorphism $H^1_{\mathrm{dR}}(X) \simeq H^2(C^{\infty}_{c}(X), \R)$.
\end{Theorem}
\begin{proof}
Let $\psi$ be a continuous Lie algebra 2-cocycle on the Poisson Lie algebra $C^{\infty}_{c}(X)$.
The restriction of $\psi$ to a 
Darboux coordinate patch $U_i$ is given by \eqref{explicit}
%$\psi(f,g) = S_{U_i}^{\mu}(g\partial_{\mu}f)$
for distributions $S_{U_i}^{\mu}$ on $U_i$
satisfying \eqref{es}.
%$\omega^{\nu\tau}\partial_{\tau}S_{U_i}^{\mu} - \omega^{\mu\tau}\partial_{\tau}S_{U_i}^{\nu} = 0$.
We define the $TX$-valued distribution $S_{U_i} \in \Omega^1_{c}(U_i)'$ by
$S_{U_i}(\alpha) = S_{U_i}^{\mu}(\alpha_{\mu})$, 
where $\al=\al_\mu dx^\mu$ on $U_i$.
Then $S_{U_i}$ agrees with $S_{U_j}$
on the overlap $U_i \cap U_j$ because there, both are determined uniquely by the expression
$S_{U_i}(gdf) = \psi(f,g) = S_{U_j}(gdf)$ for $f, g \in {C^{\infty}_{c}(U_i \cap U_j)}$.
The $TX$-valued distributions $S_{U_i}$ therefore splice together
to a $TX$-valued distribution $S \in \Omega^1_{c}(X)'$ 
such that\footnote{Note that this differs from the convention in 
\eqref{oneconvention} by a minus sign.} 
\be\label{pssi}
\psi(f,g) = S(gdf)\,.
\ee

The requirement $\omega^{\nu\tau}\partial_{\tau}S_{U_i}^{\mu} - \omega^{\mu\tau}\partial_{\tau}S_{U_i}^{\nu} = 0$
for the local distributions translates to  $S \circ \delta = 0$,
for $S\in \Om^1_c(X)'$, 
with $\delta \colon \Omega^2_{c}(X) \rightarrow \Omega^1_{c}(X)$ the canonical homology operator
of equation \eqref{canhomop}.
Indeed, using equation \eqref{PoissonCoh}, one calculates that,
in Darboux coordinates, we have
\[\delta F_{\sigma\tau} dx^{\sigma}\wedge dx^{\tau} = 
-\omega^{\nu\tau}\partial_{\tau}F_{\nu\mu} 
dx^{\mu}
+
\omega^{\mu\tau}\partial_{\tau}F_{\nu\mu} 
dx^{\nu}\,,
\]
so that for $F \in \Omega^{2}_{c}(U_i)$, we have
$S(\delta F) = \omega^{\nu\tau}\partial_{\tau}S^{\mu}_{U_i}(F_{\nu\mu}) - 
\omega^{\mu\tau}\partial_{\tau}S^{\nu}_{U_i}(F_{\nu\mu}) = 0$.
We conclude that continuous Lie algebra 2-cocycles $\psi$ on $C^{\infty}_{c}(X)$ 
correspond bijectively to 
continuous linear maps $S \colon \Omega^1_{c}(X)/\delta (\Omega^2_{c}(X)) \rightarrow \R$.
In particular, every continuous 2-cocycle
$\psi$ yields an element of 
$H^{\can}_{c,1}(X)^*$ % \simeq H^{2n-1}_{\mathrm{dR}, c}(X)^*$
by restricting $S$
to $\mathrm{Ker}(\delta)/\delta (\Omega^2_{c}(X)) = H^{\can}_{c,1}(X)$.

Note that, if $\psi=\dd H$ is the boundary of $H \in C^{\infty}_{c}(X)'$ in Lie algebra cohomology,
then, using  equation~\eqref{PoissonCoh}, we see that $\psi(f,g) = H(-\{f,g\})$ implies 
$S(gdf) = H(\delta (gdf))$.
It follows that $\psi$ is exact if and only if $S = H \circ \delta$ for some
$H \in C^{\infty}_{c}(X)'$. In particular, cohomologous cocycles 
induce the same element of 
$H^{\mathrm{can}}_{c,1}(X)^*$. 
We thus obtain a linear map 
\be\label{linm}
[\ps]\in H^2(C^{\infty}_{c}(X), \R) \mapsto S\in H^{\mathrm{can}}_{ c,1}(X)^*.
\ee
We prove that this map is injective. 
If $[\psi] \in H^2(C^{\infty}_{c}(X), \R)$ maps to zero, then
$S \colon \Omega^1_{c}(X) \rightarrow \R$ vanishes on $\mathrm{Ker}(\delta)$,
and defines a continuous linear functional 
$S \colon \Omega^{1}_{c}(X)/\mathrm{Ker}(\delta) \rightarrow \R$.
Since $\delta \colon \Omega^{1}_{c}(X) \rightarrow \mathrm{Im}(\delta) = C^{\infty}_{c,0}(X)$
is a continuous, surjective map of LF-spaces, it is open by \cite[Thm.\ 1]{DieudonneSchwartz1949}, so that
\[
\delta \colon \Omega^1_{c}(X)/\mathrm{Ker}(\delta) \rightarrow \mathrm{Im}(\delta) = C^{\infty}_{c,0}(X) 
\subseteq \Omega^0_{c}(X)
\]
is an isomorphism of topological vector spaces.
It follows that $S$ can
be written as $S = H_0 \circ \delta$
for a unique continuous functional $H_0 \colon \mathrm{Im}(\delta) \rightarrow \R$.
Since $\mathrm{Im}(\delta) \subseteq \Omega^0_{c}(X)$ is a complemented inclusion of locally convex 
spaces (one can choose a compactly supported function $f$ with $\int_{X_x} f \omega^n/n!$ on every
connected component $X_x$ of $X$), 
$H_0$ extends to a continuous functional $H \colon \Omega^0_{c}(X) \rightarrow \R$.
Since $S = H \circ \delta$, we have $[\psi] = 0$ by \eqref{pssi}.  

It remains to show that the map \eqref{linm}
%$H^2(C^{\infty}_{c}(X), \R) \rightarrow H{\mathrm{can}}_{c,1}(X)^*$
is surjective. 
Recall that  $H^{\mathrm{can}}_{c,1}(X)^*$ is isomorphic to 
$H^{2n-1}_{\mathrm{dR}, c}(X)^*$ under the symplectic hodge star operator,
which in turn is isomorphic to $H^1_{\mathrm{dR}}(X)$ by Poincar\'e duality.
%hence \eqref{linm} becomes $H^2(C^{\infty}_{c}(X), \R) \to H^1_{\mathrm{dR}}(X)$.
Proposition~\ref{Rogersmap} therefore yields a map 
$H^1_{\mathrm{dR}}(X) \rightarrow H^2(C^{\infty}_{c}(X), \R)$ in the other direction, and we show
that it is a left inverse to \eqref{linm}.
Given a class $[\alpha] \in H^1_{\mathrm{dR}}(X)$, the corresponding 
cocycle \[\psi_{\alpha}(f,g) := -\int_{X} g (i_{X_{f}}\alpha) \,\omega^n/n!\]
can also be written 
$\psi_{\alpha}(f,g) = \int_{X} \alpha \wedge *gdf$, 
as 
\begin{eqnarray*}
-g(i_{X_{f}}\alpha) \vol &=& -g\alpha \wedge i_{X_{f}}\vol\\ 
&=& \alpha \wedge gdf \wedge \omega^{n-1}/(n-1)! \\
&=& \alpha \wedge *gdf\,.
\end{eqnarray*}
Therefore, we have
$$S(\beta) = \int_{X} \alpha \wedge * \beta$$
for the  $TX$-valued distribution $S$ associated to $\psi_\al$.
Under the symplectic hodge star operator, 
the induced map $S \colon H^{\mathrm{can}}_{c,1}(X) \rightarrow \R$
corresponds with the map 
$H^{2n-1}_{\mathrm{dR},c}(X) \rightarrow \R$ given by
$[\beta] \mapsto \int_{X} \beta \wedge \alpha$.
Since this corresponds to $[\alpha] \in H^1_{\mathrm{dR}}(X)$
under Poincar\'e duality, the composition 
$H^1_{\mathrm{dR}}(X) \rightarrow 
H^2(C^{\infty}_{c}(X), \R) \rightarrow
H^1_{\mathrm{dR}}(X)$ 
is the identity, and surjectivity of \eqref{linm} follows.
We conclude that the map $[\al]\mapsto[\ps_\al]$ is an isomorphism.
\end{proof}

Using Lemma~\ref{supportcompact}, we now show that also the map~\eqref{rogernc}
is an isomorphism.

\begin{Theorem}\label{noncptpoissonh2}
For any symplectic manifold $X$, the map $[\al]\mapsto[\ps_\al]$ described in Prop.~\ref{Rogersmap}, with 
\be
\psi_{\alpha}(f,g) = \int_{X} f (i_{X_{g}}\alpha) \,\omega^n/n!,
\ee
is an isomorphism $H^1_{\mathrm{dR}, c}(X) \simeq H^2(C^{\infty}(X), \R)$.
\end{Theorem}
\begin{proof}
This is a straightforward adaptation of the proof of Theorem~\ref{HoofdZaak}.
Let $\psi$ be a continuous 2-cocycle on $C^{\infty}(X)$.
Since its restriction to $C^{\infty}_{c}(X)$ extends continuously to $C^{\infty}(X)$, 
the corresponding
$TX$-valued distribution $S$ is of compact support, $S \in \Omega^1(X)'$.
It vanishes on $\delta \Omega^2(X)$, so $S$ defines a continuous 
linear map $S \colon \Omega^1(X)/\delta \Omega^2(X) \rightarrow \R$,
hence an element of $H^{\mathrm{can}}_1(X)^*$.
If $\psi = H \circ \delta$, then $H$ is compactly supported
by Lemma~\ref{supportcompact}, so the element 
of  $H_1^{\mathrm{can}}(X)^*$ depends only on the class of $[\psi]$.
The map $H^2(C^{\infty}(X), \R) \rightarrow H_1^{\mathrm{can}}(X)^*$
is bijective by an argument similar to the one in Theorem~\ref{HoofdZaak},
with the map \eqref{rogernc} of Prop.~\ref{Rogersmap} taking the role of left inverse,
$\Omega^0(X)$ takes the role of $\Omega^0_{c}(X)$, and we use 
$H^{2n-1}_{\mathrm{dR}}(X)$
instead of $H^{2n-1}_{\mathrm{dR}, c}(X)$.
\end{proof}

%A different type of 2-cocycle $\psi_{N}$ on $C^{\infty}(X)$ is induced by a singular 
%$(2n-1)$-chain $N = \sum_{i=1}^{k}c_i \sigma_{i}$, with real coefficients 
%$c_i \in \R$ and piecewise smooth maps $\sigma_{i} \colon \Delta^{2n-1} \rightarrow X$.

Combining Proposition~\ref{kaf} with Theorem~\ref{noncptpoissonh2}, we obtain the following description 
of the second Lie algebra cohomology of $\ham(X)$.

\begin{Theorem}\label{H2Ham}
Let $X$ be a connected symplectic manifold.
Then the map from Proposition~\ref{hamcocycles},
assigning to the closed 1-form $\alpha \in \Omega^1_{c}(X)$ the cocycle 
\be
\psi_{\alpha}(X_f, X_g) = \int_{X}f\alpha(X_g)\omega^n/n!,
\ee
%Then the map described in Proposition~\ref{hamcocycles}
is an isomorphism 
\[Z^1_{ c}(X)/ d\Omega^0_{c,0} (X)\rightarrow H^2(\ham(X), \R)\,.\]
In particular, 
$H^2(\ham(X),\R)$ is isomorphic to $H^1_{\mathrm{dR}}(X)$ if $X$ is compact
and to $H^1_{\mathrm{dR}, c}(X) \oplus \R [\psi_{KS}]$ if $X$ is noncompact.
\end{Theorem}

Note that for noncompact $X$ the isomorphism between   
$Z^1_{c}(X)/ d\Omega^0_{c,0}(X) $ and $H^1_{\mathrm{dR}, c}(X) \oplus \R [\psi_{KS}]$
is not canonical.

%%%
%%%
\section{Universal central extension of $\ham(X)$}\label{s5}
%%%
%%%

Our results on the second Lie algebra cohomology of $\ham(X)$ yield an explicit classification
of the continuous central extensions of $\ham(X)$ by $\R$.
In this section, we will use this to prove that,
for a connected symplectic manifold $X$ with $H^{2n-1}_{\mathrm{dR}}(X)$
finitely generated, the central extension described in the following 
proposition is the universal central extension of $\ham(X)$.
We will then see that the same Lie algebra also yields a universal central
extension of $C^{\infty}(X)$ if $X$ is  noncompact.

\begin{Proposition}\label{beschrijving}
Let $X$ be a symplectic manifold. Then 
$\Omega^1(X)/\delta\Omega^2(X)$
is a Fr\'echet Lie algebra 
with Lie bracket
\be\label{haakje}
[[\alpha],[\beta]] :=  [\delta\alpha \cdot d\delta\beta]\,. 
\ee
If we define 
$\pi \colon \Omega^1(X)/\delta\Omega^2(X) \rightarrow %\stackrel{\pi}{\longrightarrow} 
\ham(X)$ by 
$\pi([\alpha]) := X_{\delta\alpha}$,
then 
\be\label{runiverse}
\mathfrak{z}
\stackrel{\iota}{\longrightarrow} 
\Omega^1(X)/\delta\Omega^2(X) \stackrel{\pi}{\longrightarrow} 
\ham(X) 
\ee
is a continuous central extension 
of $\ham(X)$ by 
the center
of $\Omega^{1}(X)/\delta\Omega^2(X)$, which is given by
\be\label{centreofext}
\mathfrak{z} = \mathrm{Ker}(d\circ\delta)/\delta\Omega^2(X)\,.
\ee
\end{Proposition}
\begin{Remark}
If $X$ is compact, then $\mathrm{Ker}(d\circ \delta) = \mathrm{Ker}(\delta)$, since $\delta (\Omega^1(X)) = C^{\infty}_{c,0}(X)$ by 
Lemma \ref{imde}.
It follows that $\mathfrak{z}$ equals $H^{\mathrm{can}}_{1}(X)$,
which is isomorphic to $H^{2n-1}_{\mathrm{dR}}(X)$
by Prop.~\ref{Prop:candRiso}.
If $X$ is noncompact, then $\mathfrak{z}$ is slightly bigger than $H^{2n-1}_{dR}(X)$. In fact, we have an exact sequence
\[
H^{2n-1}_{\mathrm{dR}}(X) \stackrel{*}{
\longrightarrow } 
\mathfrak{z} \stackrel{\delta}{
\longrightarrow } 
\R\,,
\]
where the extra dimension corresponds to the Kostant-Souriau extension.
\end{Remark}

\begin{Remark}\label{RkOmega2n}
We obtain a different description of the Lie algebra $\Omega^1(X)/\delta \Omega^2(X)$ 
if we identify it with $\Omega^{2n-1}(X)/d\Omega^{2n-2}(X)$
using the symplectic hodge star operator.
The Lie bracket on $\Omega^{2n-1}(X)/d\Omega^{2n-2}(X)$ takes the form
\[
[[\ga_1],[\ga_2]]=[f_1 df_2\wedge\om^{n-1}/(n-1)!],
\]
where the function $f \in C^\oo(X)$ is uniquely defined by the form
$\ga\in\Om^{2n-1}(X)$ via $f\vol=d\ga$.
The Lie algebra homomorphism $\Omega^{2n-1}(X)/d\Omega^{2n-2}(X)\rightarrow \ham(X)$ is given by 
$[\gamma] \mapsto X_{f}$, yielding the central extension 
\be\label{uce2}
%0 \to %H^{2n-1}_{\dR}(X)
\mathrm{Ker}(\delta \circ d)/d\Omega^{2n-2}(X)
\to\Om^{2n-1}(X)/d\Om^{2n-2}(X)\to\ham(X)
%\to 0.
\ee
If $X$ is compact, then $\mathrm{Ker}(\delta \circ d)/d\Omega^{2n-2}(X) = H^{2n-1}_{\dR}(X)$.
\end{Remark}

\begin{proof}
To show that $\Omega^1(X)/\delta\Omega^2(X)$ is Fr\'echet, it suffices to exhibit  
$\delta\Omega^{2}(X)$ as a closed subspace of the Fr\'echet space $\Omega^1(X)$
(cf.~\cite[Thm~1.41]{Rudin1991}).
Since $\delta = (-1)^{k+1}*d\,*$ and $*^2 = \mathrm{Id}$, 
the equality $\alpha = \delta F$ for $F\in\Om^2(X)$ is equivalent to
$*\alpha = d *F$, so that 
$\alpha \in \Omega^{1}(X)$ is $\delta$-exact
if and only if $\int_{X} (*\alpha) \wedge \beta = 0$ for all closed $\beta \in \Omega^1_{c}(X)$.
Since this is a closed condition, 
$\delta\Omega^{2}(X)$ is closed in $\Omega^1(X)$.

The bracket \eqref{haakje} 
is continuous because it is given in terms of differential operators.
Moreover, $\pi$ preserves the brackets because
by %\eqref{PoissonCoh} 
\eqref{rkform1}, we have 
\be\label{projectiescherm}
\delta [[\alpha],[\beta]] = \delta(\delta\alpha \cdot d\delta\beta) = \{\delta \alpha, \delta\beta\}\,,
\ee
so that $X_{\delta [[\alpha],[\beta]]} = [X_{\delta\alpha},X_{\delta\beta}]$.

To check that \eqref{haakje} defines a Lie bracket, we use the following identity
proven in Lemma \ref{lem1}:
\be\label{dindelta}
d\Omega^0(X) \subseteq \delta \Omega^{2n-2}(X)\,.
\ee
Antisymmetry follows from 
\[[[\alpha],[\beta]] + [[\beta],[\alpha]] = [d(\delta\alpha \cdot \delta\beta)] = 0\,.\]
For the Jacobi identity note that, by \eqref{projectiescherm}, the jacobiator equals to
\[
\sum_{\mathrm{cycl}}[[[\alpha_1],[\alpha_2]],[\alpha_3]] = 
\Big[\sum_{\mathrm{cycl}} \{\delta\alpha_1,\delta\alpha_2\}\cdot d\delta\alpha_3\Big]\,.
\]
This is zero by \eqref{dindelta}, because by 
%\eqref{PoissonCoh} 
\eqref{rkform2} we have
\begin{eqnarray*}
\delta (\delta\alpha_1 \cdot d\delta\alpha_2 \wedge d\delta \alpha_3) &=&
\{\delta\alpha_1, \delta\alpha_2\} d\delta \alpha_3
-\{\delta\alpha_1, \delta\alpha_3\} d\delta \alpha_2
-\delta\alpha_1 d\{\delta\alpha_2, \delta\alpha_3\}\\
&=&
\sum_{\mathrm{cycl}}\{\delta\alpha_1, \delta\alpha_2\} d\delta \alpha_3
-d(\de\al_1\{\de\al_2,\de\al_3\})\,.
%\{\delta\alpha_1, \delta\alpha_2\} d\delta \alpha_3
%+\{\delta\alpha_3, \delta\alpha_1\} d\delta \alpha_2
%+\{\delta\alpha_2, \delta\alpha_3\}d\delta\alpha_1\,.
\end{eqnarray*}

The Lie algebra homomorphism $\pi$ is surjective by Lemma \ref{imde}.
The kernel of $\pi$ consists of those 
$[\alpha]$ in $\Omega^1(X)/\delta\Omega^2(X)$ for which $\delta\alpha$ is locally constant, 
hence
$\mathrm{Ker}(\pi) = \mathrm{Ker}(d\circ\delta)/\delta\Omega^2(X)$.

The Lie algebra $\ham(X)$ has trivial center, because 
$[X_{f},X_{g}] = 0$ for all $X_g\in \ham(X)$ implies that
$i_{X_{f}}dg $ is constant for all $g\in C^{\infty}(X)$, hence $X_{f} = 0$.
Since  $\pi$ maps
the center $\mathfrak{z}$ of $\Omega^1(X)/\delta\Omega^2(X)$ 
to the (trivial) center of $\ham(X)$, we have $\mathfrak{z} \subseteq \mathrm{Ker}(\pi)$.
Since $d\circ\delta\alpha = 0$ implies that $[\alpha] \in \Omega^{1}(X)/\delta\Omega^2(X)$
is central, the converse inclusion follows. 
\end{proof}

 \begin{Proposition}\label{extisperfect}
 Let $X$ be a symplectic manifold. Then
 $\Omega^1(X)/\delta\Omega^2(X)$, equipped with the Lie bracket \eqref{haakje}, is a 
 perfect Lie algebra.
 % for 
 %any symplectic manifold $X$.
 \end{Proposition}
 \begin{proof}
 First, suppose that $X$ is connected, and $[\alpha] \in \Omega^1(X)/\delta\Omega^2(X)$ is of the form
 $[\alpha] = [fdg]$, with $f, g \in C^{\infty}(X)$.
If $X$ is noncompact, then $\delta:\Om^1(X)\to C^\oo(X)$ 
is surjective by Lemma \ref{imde} and we can
 choose $\phi, \gamma \in \Omega^{1}(X)$ with $\delta\phi = f$ and $\delta \gamma = g$, so that $[\alpha] = [\delta\phi \cdot d\delta \gamma] = [[\phi],[\gamma]]$ is a commutator.
If, on the other hand, $X$ is compact, then $\mathrm{Im}(\delta)$ is equal to
$C^\oo_{c,0}(X)$, the space of zero integral functions,  by Lemma \ref{imde}.
If we write $f = f_0 + c_f$ 
with $f_0 \in C^\oo_{c,0}(X)$ and $c_f \in \R$, and similarly
$g = g_0 + c_g$,  then $[fdg] = [(f_0+ c_f) dg_0] = - [df_0 g_0] = [f_0dg_0]$.
Since we can choose $\phi, \gamma \in \Omega^1(X)$ with $f_0 = \delta\phi$ and $g_0 = \delta\gamma$,
we see that 
$[\alpha] = [\delta\phi \cdot d\delta\gamma] =[[\phi],[\gamma]]$ is a commutator.

It suffices to show, then, that $\alpha \in \Omega^1(X)$ can be written as a finite sum 
of elements of the form $fdg$.
As in the proof of Prop.~\ref{Prop:noncpspoisson}, we use dimension theory
(\cite[Thm.~V1]{Hurewicz41}) to find
a cover $U_{k,r}$ of $X$ 
by open Darboux coordinate patches $U_{k,r}$,
where
$k \in \mathbb{N}$ is a countable index, $r \in \{1,\ldots,2n+1\}$ is a finite index,
and $U_{k,r} \cap U_{k',r} = \emptyset$ for all $k\neq k'$.
Using a partition of unity, we write 
$\alpha = \sum_{k=1}^{\infty}\sum_{r=1}^{2n+1}\alpha^{k,r}$
with $\alpha^{k,r} \in \Omega^1_{c}(U_{k,r})$.
The intersection properties of the $U_{k,r}$ ensure that  
in a single point $x\in X$, the sum has at most $2n+1$ nonzero terms.
In local coordinates $x^{\mu}$, we have $\alpha^{k,r} = \sum_{\mu=1}^{2n} \alpha^{k,r}_{\mu}dx^{\mu}$.
Choosing compactly supported functions $\xi^{k,r,\mu}$ on $U_{k,r}$ that agree with $x^{\mu}$ 
on the support of $\alpha^{k,r}_{\mu}$, we see that $\alpha^{k,r} = \sum_{\mu=1}^{2n}\alpha^{k,r}_{\mu} d \xi^{k,r, \mu}$
is of the required form.
If we set $f^{r}_{\mu} := \sum_{k\in \N} \alpha^{k,r}_{\mu}$ and $g^{r,\mu} := \sum_{k\in \N} \xi^{k,r,\mu}$,
then because $\alpha^{k,r}_{\mu}d\xi^{k',r,\nu} = 0$ for $k \neq k'$, we have 
\[
\sum_{r=1}^{2n+1}\sum_{\mu = 1}^{2n} f^{r}_{\mu} dg^{r,\mu} = 
\sum_{r=1}^{2n+1}\sum_{\mu = 1}^{2n}\sum_{k\in \N}  \alpha^{k,r}_{\mu}d\xi^{k,r,\mu} = 
\sum_{r=1}^{2n+1}\sum_{k\in \N} \alpha^{k,r} =
\alpha\,.
\]
It follows that for $X$ connected, $[\alpha]$ is a sum of at most $2n(2n+1)$ commutators
$[\phi^{r}_{\mu}]$ and $[\gamma^{r,\mu}]$, with $\delta \phi^{r}_{\mu} = f^{r}_{\mu}$ and 
$\delta \gamma^{r,\mu} = g^{r,\mu}$ if $X$ is noncompact, and
$\delta \phi^{r}_{\mu} = f^{r}_{\mu 0}$ and $\delta \gamma^{r,\mu} = g^{r,\mu}_{0}$ if $X$ is compact. 

If $X$ is not connected, then we find $\phi^{r,x}_{\mu}$ and $\gamma^{r,\mu, x}$ as above for every 
connected component $X_{x}$ of $X$, 
so that on $X_x$, we have 
\[[\alpha|_{X_x}] = \sum_{r=1}^{2n+1}\sum_{\mu = 1}^{2n} [\delta\phi^{r,x}_{\mu} d\delta \gamma^{r,\mu,x}]\,.\]
If we set $\phi^{r}_{\mu} = \sum_{x} \phi^{r,x}_{\mu}$ and 
$\gamma^{r,\mu} := \sum_{x} \gamma^{r,\mu,x}$, then 
since $\delta\phi^{r,x}_{\mu} d\delta \gamma^{r,\mu,x'} = 0$ for $X_{x} \neq X_{x'}$, we have
\[[\alpha] = \sum_{r=1}^{2n+1}\sum_{\mu = 1}^{2n} [[\phi^{r}_{\mu}], [\gamma^{r,\mu}]]\]
as a sum of at most $2n(2n+1)$ commutators.
\end{proof}
 \begin{Remark}
 It follows from the previous proof that if $X$ is of dimension $2n$, then every element of $\Omega^1(X)/\delta\Omega^2(X)$
 can be written as a sum of at most $2n(2n+1)$ commutators, cf. Rk.~\ref{Rk:numbercomm}.
 \end{Remark}

Recall from Sec.~\ref{sec:contcohomo} 
that a linearly split central extension is called \emph{universal}
for $\mathfrak{a}$ if it maps to every linearly split central $\mathfrak{a}$-extension, 
and \emph{universal} if it is $\mathfrak{a}$-universal
for every locally convex space $\mathfrak{a}$.

\begin{Theorem}\label{universalHam}
Let $X$ be a connected symplectic manifold. 
Then the continuous central extension
\be\label{runiverseII}
\mathfrak{z}
%:=
%\mathrm{Ker}(d\circ\delta)/\delta\Omega^2(X)
\stackrel{\iota}{\longrightarrow} 
\Omega^1(X)/\delta\Omega^2(X) \stackrel{\pi}{\longrightarrow} \ham(X)
\ee
of Fr\'echet Lie algebras (cf.\ Prop.~\ref{beschrijving})
is linearly split, and $\mathfrak{a}$-universal for finite dimensional vector spaces $\mathfrak{a}$.
If, moreover, $H^{2n-1}_{\mathrm{dR}}(X)$ is finitely generated, then it %\eqref{runiverseII}
is  the universal central extension of $\ham(X)$. %, unique up to isomorphism.
\end{Theorem}
%%%
\begin{proof}
By Prop.~\ref{extensionvscohomology} and Thm.~\ref{H2Ham}, every 
central $\R$-extension of $\ham(X)$ is isomorphic to 
$\R \oplus_{\psi_{\alpha}}\ham(X)$,
with Lie bracket 
\be\label{exthaakje}
\big[ (a,X_{f}), (b, X_{g})\big] = \Big(\psi_{\alpha}(X_f, X_g), X_{\{f,g\}} \Big)\,.
\ee
The class $[\psi_{\alpha}]\in H^2(\ham(X),\R)$ is determined by
$\alpha \in Z^1_{ c}(X)/d\Omega^{0}_{c, 0}(X)$, cf.\ eqn.~\eqref{voorham}.
%Using that $i_{X_{g}}(f\alpha\wedge \omega^{n}/n!) = 0$ implies 
Using 
\[
f\cdot(i_{X_{g}}\alpha)\omega^n/n! \stackrel{\eqref{kleintje}}{=} -\alpha\wedge(fdg)\wedge \omega^{n-1}/(n-1)!
\stackrel{\eqref{ster}}{=}-\alpha \wedge *(fdg)\,,
\]
we rewrite the cocycle $\ps_\al$ 
%from equation \eqref{voorham} 
as \[\psi_{\alpha}(X_{f},X_{g})=\int_X f\cdot(i_{X_g}\al)\vol = S_{\alpha}(fdg)\,,\]
with $S_{\alpha} \colon \Omega^{1}(X)\rightarrow \R$ given by 
\[
S_{\alpha}(\beta) := -\int_{X}\alpha \wedge *\beta\,.
\]
Note that $S_{\alpha}$ factors through a functional
$\overline{S}_{\alpha} \colon \Omega^{1}(X)/\delta\Omega^2(X)\rightarrow \R$,
because for $F \in \Omega^2(X)$ we have
$\int_{X}\alpha \wedge * \delta F = \int_{X}\alpha \wedge d * F
=\int_{X}d(\alpha \wedge *F) = 0$.

For every central $\R$-extension of $\ham(X)$, we thus 
obtain a continuous linear map  
\[\phi \colon \Omega^{1}(X)/\delta\Omega^2(X) \rightarrow \R \oplus_{\psi_{\alpha}} \ham(X)\]
defined by
\be \label{eq:mapto1d}
\phi([\beta]) := (\overline{S}_{\alpha}([\beta]), \pi(\beta))\,.
\ee
This is a Lie algebra homomorphism because 
 $[[\beta],[\beta']] = [\delta\beta\cdot d\delta\beta']$ by definition of the Lie bracket on $\Omega^{1}(X)/\delta\Omega^2(X)$, so
\[\overline{S}_{\alpha}([[\beta],[\beta']]) 
=S_\al(\de\beta\cdot d\de\beta')
= \psi_{\alpha}(X_{\delta\beta},X_{\delta\beta'})\,.
\]
The homomorphism $\phi$ is unique with the property that the $\ham(X)$-valued component
coincides with $\pi$.
Indeed, if
$\phi' \colon \Omega^{1}(X)/\delta\Omega^2(X) \rightarrow \R \oplus_{\psi_{\alpha}} \ham(X)$
is another such homomorphism, then 
$\phi - \phi'$ is a 
Lie algebra homomorphism $\Omega^{1}(X)/\delta\Omega^2(X) \rightarrow \R \oplus \{0\}$, hence trivial
because $\Omega^{1}(X)/\delta\Omega^{2}(X)$ is perfect (cf. Prop.~\ref{extisperfect}).
This shows that the continuous central extension \eqref{runiverseII} is $\R$-universal,
hence universal for finite dimensional vector spaces $\mathfrak{a}$ by \cite[Lemma~1.12]{Neeb03}.

A linear splitting of \eqref{runiverseII} exists because the de Rham
complex $(\Omega^{\bullet}(X), d)$, and hence the isomorphic canonical complex 
$(\Omega^{\bullet}(X), \delta)$,
is split and cosplit as a complex of Fr\'echet spaces \cite[Prop.~5.4]{Palamodov1972}.
A linear splitting $\ham(X) \rightarrow C^{\infty}(X)$ corresponds to a splitting of the canonical
complex 
in degree $0$, and a splitting $C^{\infty}(X) \rightarrow \Omega^1(X)$ 
to a splitting in degree 1.
If $X$ is compact, then a splitting of the de Rham complex can also be obtained using Hodge theory.

If $H^{2n-1}_{\dR}(X)$ is finite dimensional, then so is $H^2(\ham(X), \R)$, so a universal central extension with 
finite dimensional kernel
exists by \cite[Lem.~2,7, Cor.~2.13]{Neeb03}. 
Since this universal central extension and the extension \eqref{runiverseII} 
are both linearly split as well as $\mathfrak{a}$-universal for finite dimensional $\mathfrak{a}$,
they are isomorphic, and \eqref{runiverseII} is universal. 
\end{proof}
\begin{Remark}
If we identify $\Omega^1(X)/\delta\Omega^2(X)$ with $\Omega^{2n-1}(X)/d\Omega^{2n-2}(X)$ 
according to Remark~\ref{RkOmega2n},
the morphism \eqref{eq:mapto1d} 
to the central extension $\RR\oplus_{\psi_\al}\ham(X)$ 
with cocycle $\psi_\al$ defined in \eqref{voorham} takes the particularly simple form
\[
{}[\ga]\mapsto(-\int_X\al\wedge\ga,X_f)\,,
\]
where the function $f$ is given by $f\vol=d\ga$.
\end{Remark}

From the above theorem, one readily derives the universal central extension of the Poisson 
algebra $C^{\infty}(X)$ for a noncompact connected manifold $X$.
If $X$ is compact, no universal central extension exists 
because  $C^{\infty}(X)$ % \simeq \ham(X)\oplus \R$ 
is not perfect (cf. Prop.~\ref{ALDM}).

\begin{Corollary}\label{Cor:univpois}
Let $X$ be a noncompact connected symplectic manifold. Then the continuous central extension 
\[ H^{\mathrm{can}}_{1}(X) \rightarrow \Omega^1(X)/\delta\Omega^2(X)\stackrel{\delta}{\longrightarrow} C^{\infty}(X)\] 
is linearly split and $\mathfrak{a}$-universal for finite dimensional spaces $\mathfrak{a}$.
If, moreover, $H^{\mathrm{can}}_{1}(X) \simeq H^{2n-1}_{\mathrm{dR}}(X)$ is finitely generated,
then it is the universal central extension of $C^{\infty}(X)$.
\end{Corollary}
\begin{proof}
By eqn.\ \eqref{nulopeen}, %the argument in the proof of Prop.~\ref{kaf}, 
every $\mathfrak{a}$-valued 2-cocycle $\psi$ on $C^{\infty}(X)$ vanishes on the constant functions.
It follows that every linearly split continuous central extension 
$\mathfrak{a} \rightarrow \fg^{\sharp} \stackrel{\mathrm{pr}}{\longrightarrow} C^{\infty}(X)$ 
of $C^{\infty}(X)$ defines, by concatenation with the map $C^{\infty}(X) \rightarrow \ham(X)$ 
that takes a hamilton function to the corresponding hamiltonian vector field, 
a linearly split continuous central extension of
$\ham(X)$. 
If either $\mathfrak{a}$ or $H^{2n-1}_{\mathrm{dR}}(X)$ is finite dimensional, then 
Thm.~\ref{universalHam} yields a unique
continuous homomorphism $\phi \colon {\Omega^1(X)/\delta\Omega^2(X) \rightarrow \fg^{\sharp}}$
such that the big triangle in the following diagram commutes:
\begin{center}
\centerline{
\xymatrix{
&&\Omega^{1}(X)/\delta\Omega^2(X) \ar[ld]_{\phi}\ar[d]^{\delta}\ar[rd]^{\pi}&\\
\mathfrak{a}\ar[r]&\fg^{\sharp}\ar[r]^{\mathrm{pr}}&C^{\infty}(X)\ar[r]&\ham(X)\,.&
}
}
\end{center}
Since $\delta$ is the unique map which makes the right hand triangle commute, 
we have $\mathrm{pr} \circ \phi = \delta$. For noncompact $X$, 
the cokernel
$H^{\mathrm{can}}_{0}(X) \simeq H^{2n}_{\mathrm{dR}}(X)$
of the Lie algebra homomorphism 
$\delta \colon \Omega^1(X)/\delta\Omega^{2}(X) \rightarrow C^{\infty}(X)$ vanishes,
so $\delta$ is surjective with kernel $H^{\mathrm{can}}_{1}(X)$.
\end{proof}

\begin{Remark}
The Lie algebra of hamiltonian vector fields $\ham(X)$ is perfect
for any symplectic manifold $X$.
Indeed, it is the image by $\pi$ of the perfect Lie algebra $\Omega^1(X)/\delta\Omega^2(X)$
by the Lie algebra homomorphism $\pi$.
The Poisson Lie algebra $C^\oo(X)$ in the noncompact case
and $C^\oo_0(X)$ in the compact case
are also perfect, as images of the perfect Lie algebra $\Omega^1(X)/\delta\Omega^2(X)$
by the Lie algebra homomorphism $\de$.
Thus we recover Corollary \ref{hamp}, as well as the `perfect part' of 
Corollary \ref{cptsuppfirstcohomology} and Proposition \ref{Prop:noncpspoisson}.
\end{Remark}

%\todo{Shall we delete some of these proofs from section 3?}

\paragraph{Singular cocycles and Roger cocycles revisited.}

The universal central extension $\Omega^1(X)/\delta\Omega^2(X)$ of $\ham(X)$
yields a straightforward way to decide when the Roger cocycles $\psi_{\alpha}$ 
and singular cocycles  $\psi_{N}$ defined on page \pageref{RogerandSingular}
 are cohomologous.

By the universal property of (\ref{runiverse}) for a connected, symplectic manifold $X$,
%$\Omega^{1}(X)/\delta\Omega^2(X)$, 
continuous 2-cocycles $\psi_{}$ of $\ham(X)$
correspond bijectively to continuous linear functionals $S$
on $\Omega^{1}(X)/\delta\Omega^2(X)$
by
\[
\psi_{S}(X_{f},X_{g}) = S([fdg])\,,
\]
and two cocycles are cohomologous, i.e. $\psi_{S} \sim \psi_{S'}$, if and only if
$S - S'$
vanishes on the center~$\mathfrak{z}= \mathrm{Ker}(d\circ\delta)/\delta\Omega^2(X)$.

The singular 2-cocycles $\psi_{N}$  come from $(2n-1)$-cycles 
$N = \sum_{i=1}^{k} c_{i}\sigma_{i}$ in the singular homology of $X$, 
with real coefficients $c_i$ and 
piecewise smooth maps $\sigma_{i} \colon \Delta^{2n-1} \rightarrow X$. 
The functional $S_{N}([\beta]) = \int_{N} *\beta$
on $\Omega^1(X)/\delta\Omega^2(X)$ yields the singular 2-cocycle
\be\label{singcocycle}
	\psi_{N}(X_{f}, X_{g}) =S_N([fdg])= \int_{N} fdg \wedge \omega^{n-1}/(n-1)!\,.
\ee

If $X$ is compact, then $\psi_{N} \sim \psi_{N'}$ if and only if
$N \sim N'$, % in $H_{2n-1}(X,\R)$, 
in the sense that
$N - N'= \partial M$ 
for a singular $2n$-chain $M$ in $X$ with real coefficients. 
Indeed, $\psi_{N} \sim \psi_{N'}$ if and only if $\int_{N-N'} *\beta =0$ 
for all $\beta \in \Omega^{1}(X)$ with $d\delta \beta = 0$. For $X$ compact, $d\delta \beta = 0$ is equivalent to 
$d(*\beta) = 0$, so by Poincar\'e duality, $\psi_{N} \sim \psi_{N'}$ if and only if 
$N - N' = \partial M$.
In this case we get  an alternative description of
$H^2(\ham(X),\R)$ in terms of singular homology.
The correspondence $[N]\mapsto[\psi_N]$ yields and isomorphism between 
$H_{2n-1}(X,\R)$ and $H^2(\ham(X),\R)$.

If $X$ is noncompact, then $\psi_{N} \sim \psi_{N'}$ if and only if 
$N - N'= \partial M$, with the additional requirement that $M$ has zero 
symplectic volume, $\int_{M} \vol = 0$.
Indeed, for $X$ noncompact, $d\delta \beta = 0$ is equivalent to 
$d(*\beta) = c\vol$ for some $c \in \R$. From $c=0$, we find 
$N - N' = \partial M$, and $c\neq 0$ yields the further requirement 
$\int_{M} \vol = 0$.

The Roger 2-cocycles $\psi_{\alpha}$ come from  compactly supported, closed 1-forms $\alpha \in \Omega_{c}^1(X)$,
\be
\psi_{\alpha}(X_f,X_g) = \int_{X} f(i_{X_g}\alpha) \vol\,,
\ee
via the functionals $S_{\alpha}([\beta]) := \int_{X}*\beta \wedge \alpha$.
%With the same type of arguments as above, one gets that
Two such cocycles are cohomologous, i.e. $\psi_{\alpha} \sim \psi_{\alpha'}$, if and only if $\alpha - \alpha' = dh$ with $\int_{X} h \vol = 0$. 
This yields the isomorphism $H^2(\ham(X),\R) \simeq Z^1_{c}(X)/d\Omega^0_{c,0}(X)$
of Theorem \ref{H2Ham}.

In the same vein, we have $\psi_{\alpha} \sim \psi_{N}$ if and only if 
$\int_{N} \gamma = \int_{X}\alpha\wedge \gamma$ 
for every $\gamma \in \Omega^{2n-1}(X)$ with 
$\delta d \gamma = 0$.
In the compact case, %this is equivalent to $d\gamma = 0$,
this reduces to
the requirement is that $[N] \in H_{2n-1}(X,\R)$ be Poincar\'e dual to
 $[\alpha] \in H^1_{\mathrm{dR},c}(X)$.

%%%
%%%
\section{Central extensions of $\smp(X)$}\label{s6}
%%%
%%%

Let $X$ be a connected symplectic manifold.
In order to determine $H^2(\smp(X), \R)$, we observe that 
$\ham(X)$ is a perfect, closed, complemented ideal in $\smp(X)$
with abelian quotient $\smp(X)/\ham(X) \simeq H^1_{\mathrm{dR}}(X)$.

\subsection{A 5-term exact sequence}\label{sec:5term}

If $\fh$ is a perfect, closed, complemented ideal in a Fr\'echet Lie algebra $\fg$,
then it is shown in \cite[Thm.~2.3]{Vizman2006} that the exact sequence
\begin{equation}\label{fgh}
0 \rightarrow \fh \stackrel{\iota}{\longrightarrow} \fg 
\stackrel{p}{\longrightarrow} \fg/\fh \rightarrow 0
\end{equation}
of Fr\'echet Lie algebras gives rise to a 5-term exact sequence
\be\label{5term}
0 \rightarrow H^2(\fg/\fh , \R)
\stackrel{p^*}{\rightarrow} H^2(\fg, \R)
\stackrel{\iota^*}{\rightarrow} H^2(\fh, \R)^{\fg}
\stackrel{T}{\longrightarrow} H^3(\fg/\fh, \R) 
\stackrel{p^*}{\rightarrow} H^3(\fg, \R)
\ee
in continuous Lie algebra cohomology.

The transgression map 
$T \colon H^2(\fh, \R)^{\fg}
\rightarrow H^3(\fg/\fh, \R)$ is defined as follows.
If $[\psi] \in H^2(\fh, \R)$ is $\fg$-invariant, then for each $v\in \fg$, 
there exists a unique $\theta_v \in \fh'$ such that $L_{v}\psi = \dd \theta_{v}$.
%It is shown in \cite[Lemma~2.1]{Vizman2006} that 
Moreover, the map $\fg \rightarrow \fh' $ given by $ v \mapsto \theta_v$
is linear, and the corresponding map $\theta \colon \fg \times \fh \rightarrow \R$ continuous  \cite[Lemma~2.1]{Vizman2006}.
If $\psi' \colon \fg \times \fg \rightarrow \R$ is a skew-symmetric continuous extension of $\theta$,
then $\dd\psi'$ induces a 3-cocycle $\ol{\dd \psi'}$ on $\fg/\fh$, whose class 
$[\ol{\dd \psi'}] \in H^3(\fg/\fh)$ does not depend on the choice of $\psi'$.
The transgression is defined as $T([\psi]) := [\ol{\dd \psi'}]$.

\begin{Remark}
Since $H^1(\h,\R)=0$, the transgression map $T$ coincides with the map $d_3$ %$d_3:E_3^{0,2}\to E_3^{3,0}$ 
for the Hochschild-Serre spectral sequence associated to the filtration 
\[F^pC^{p+q}(\g)=
\{\si\in C^{p+q}(\g):i_{H_1}\dots i_{H_{q+1}}\si=0, \text{ for all } H_i\in\h\}\,.
\]
\end{Remark}

From \eqref{5term} with $\fh = \ham(X)$ and $\fg = \smp(X)$, we find 
\begin{equation}\label{kert}
H^2(\smp(X),\R) \simeq {\textstyle \bigwedge^2} H^1_{\mathrm{dR}}(X)^* \oplus \mathrm{Ker}(T)\,,
\end{equation}
where the map $\bigwedge^2 H^1_{\mathrm{dR}}(X)^* \rightarrow H^2(\smp(X),\R)$
takes $\gamma \in \bigwedge^2 H^1_{\mathrm{dR}}(X)^*$
to the class of $\psi_{\gamma}(v,w) := \gamma([i_v\omega], [i_w\omega])$.
For example, if $X$ is compact, the alternating pairing
$\gamma([\alpha], [\beta]) = \int_{X} \alpha \wedge \beta \wedge \omega^{n-1}/(n-1)!$
gives rise to the Lie algebra cocycle described in \cite{IsmagilovLosikMichor2006}.

\subsection{The $\smp(X)$-invariants in $H^2(\ham(X),\R)$}

In order to determine the remaining part $\mathrm{Ker}(T)$ of the second cohomology group of $\smp(X)$,
we must first determine the $\smp(X)$-invariant part of $H^2(\ham(X), \R)$, 
and then calculate the transgression map $T$.
Recall from Theorem~\ref{H2Ham} that every class $[\psi] \in H^2(\ham(X), \R)$
has a representative of the form
\[
\psi_{\alpha}(X_f, X_g) = \int_X f \alpha(X_g)\vol
\]
for a closed compactly supported 1-form $\alpha \in \Omega^1_{c}(X)$, and that $[\psi_{\alpha}] = [\psi_{\alpha'}]$
if and only if $\alpha - \alpha' = dh$ with $\langle h \rangle = 0$, where
\be\label{average}
\langle h \rangle : = \int_{X} h \vol\,.
\ee
\begin{Proposition}\label{slartibartfast}
Let $X$ be a connected symplectic manifold. Then
the action of the symplectic vector fields $\smp(X)$ on $H^2(\ham(X), \R)$
is given by
\[
L_v[\psi_\alpha] = - \langle i_v \alpha \rangle [\psi_{KS}]\,,
\]
where $v$ is an element of $\mathrm{sp}(X)$ and 
$[\psi_{KS}]$ is the Kostant-Souriau class \eqref{eq:KScocycle}.
\end{Proposition}

\begin{proof}
First we rewrite the cocycle $\ps_\al$, using \eqref{kleintje}, as
\[
\psi_{\alpha}(X_f, X_g) = \int_{X} f dg \wedge \alpha \wedge \omega^{n-1}/(n-1)!\,.
\]
It follows that $(L_{v}\psi_{\alpha})(X_f,X_g) := - \psi_{\alpha}(L_v X_f, X_g) 
- \psi_\al(X_f, L_v X_g)$
satisfies for all $v\in\smp(X)$
\begin{eqnarray*}
(L_{v}\psi_{\alpha})(X_f,X_g) &=& -\int_{X} L_v(fdg) \wedge \alpha \wedge \omega^{n-1}/(n-1)! \\
%& = & - \int_{X} fdg \wedge L_{v}\big(\alpha \wedge \omega^{n-1}\big)/(n-1)!\\
&=&  \int_{X} fdg \wedge d\big(\alpha(v) \omega^{n-1}\big)/(n-1)!\\
& = & - \int_{X} \alpha(v) df \wedge dg \wedge \omega^{n-1}/(n-1)!\,,\\
\end{eqnarray*}
so that by equation \eqref{commutatorexact}, we have
\be\label{grondeekhoorn}
(L_{v}\psi_{\alpha})(X_f,X_g) = - \int_{X} \alpha(v) \{f,g\} \omega^{n}/n!\,.
\ee
For $x\in X$ and $\alpha \in \Omega^1_{c}(X)$ closed, we define 
the continuous map $\theta_{\alpha} \colon \smp(X) \rightarrow \ham(X)'$
by
\begin{equation}\label{teta-non}
\theta_\al(v)(X_f) := \int_X\alpha(v)(f-f_x)\omega^{n}/n!,
\end{equation}
where $f - f_x$ is the unique Hamiltonian function of $X_f$ that vanishes at $x$.
Since the coboundary of $\theta_{\alpha}(v)$ is a linear combination of $L_v\psi_\alpha$ and $\psi_{KS}$, namely
\begin{eqnarray}\label{worteltaart}
\dd(\theta_\al(v))(X_f,X_g) & = & -\theta_\al(v)(X_{\{f,g\}}) \\
& = & - \int_{X}\alpha(v) (\{f,g\} - \{f,g\}_{x})\omega^n/n!\nonumber\\
& = & \langle\alpha(v)\rangle \psi_{KS}(X_f, X_g) + L_{v}\psi_{\alpha}(X_f,X_g)\,, \nonumber
\end{eqnarray}
we have $L_v[\psi_\alpha]=- \langle \alpha(v)\rangle [\psi_{KS}]$.
\end{proof}
%Since $[\psi_{KS}] = 0$ if and only if $X$ is compact, 
This reveals a dichotomy between the compact and noncompact case:
if $X$ is compact, then $[\psi_{KS}]$ is zero by Corr.~\ref{corr:KSextension}, 
so every class in $\ham(X)$ is 
automatically $\smp(X)$-invariant. 

\begin{Corollary}\label{spinv}
Let $X$ be a compact connected symplectic manifold. Then 
$\smp(X)$ acts trivially on
the continuous second cohomology 
$H^2(\ham(X),\R)$. 
\end{Corollary}
If, on the other hand, $X$ is noncompact, then the $\smp(X)$-invariant part 
of $H^2(\ham(X), \R)$ can be quite small.
Indeed, in terms of the alternating pairing
\[
(\,\cdot\, ,\,\cdot\, ) \colon H^1_{\mathrm{dR,c}}(X)\times H^1_{\mathrm{dR}}(X)\to\R
\]
defined by
\begin{equation}\label{pair-nonII}
([\alpha], [\beta]) := \int_X \alpha \wedge \beta \wedge \omega^{n-1}/(n-1)!\,,
\end{equation}
we have the following result.
\begin{Corollary}\label{spinv-non}
Let $X$ be a noncompact connected symplectic manifold. Then the class
$[\psi_{\alpha}] \in H^2(\ham(X),\R)$ is annihilated by $\smp(X)$
if and only if $[\alpha] \in H^1_{\mathrm{dR}, c}(X)$ is in the kernel of the 
alternating 
%symplectic 
pairing \eqref{pair-nonII}.
In particular, we have $\R [\psi_{KS}] \subseteq H^2(\ham(X),\R)^{\smp(X)}$.
\end{Corollary}
\begin{proof}
Since $[\psi_{KS}]$ is nonzero 
by Corollary~\ref{corr:KSextension},
Proposition~\ref{slartibartfast} shows that
$[\psi_{\alpha}]$ is %$\smp(X)$-invariant if and only if
annihilated by $\smp(X)$ if and only if  
$\langle \alpha(v) \rangle$ vanishes for all $v\in \smp(X)$.
%Since ${i_{v} (\alpha \wedge \omega^{n})/n!} = 0$, we have 
%$\alpha(v)\omega^n/n! = \alpha \wedge i_{v}\omega \wedge \omega^{n-1}/(n-1)!$, 
Because $\langle \alpha(v) \rangle = ([\alpha], [i_v\omega])$ by
\eqref{kleintje} and $v \mapsto [i_v\omega]$ 
is surjective onto $H^1_{\mathrm{dR}}(X)$, 
we have that $\langle \alpha(v) \rangle = 0$ for all $v \in \mathrm{sp}(X)$ if and only if 
$[\alpha]$ is in the kernel of the alternating pairing \eqref{pair-nonII}.

Since $[\psi_{KS}]$ is equal to $[\psi_{dh}]$ for $h\in \Omega^0_{c}(X)$ with $\langle h \rangle = 1$,
$[\alpha] = [dh] = 0$ is in the kernel, so $[\psi_{KS}]$ is always $\smp(X)$-invariant. 
\end{proof}

\begin{Remark}
By definition, a symplectic manifold is Lefschetz if the map
\[H^1_{\mathrm{dR}}(X) \rightarrow H^{2n-1}_{\mathrm{dR}}(X) \colon [\alpha] \mapsto [\omega^{n-1} \wedge \alpha]\] is an isomorphism.
Since the intersection pairing is then nondegenerate, 
it follows that for a 
noncompact Lefschetz manifold $X$,
the only invariant class is $[\psi_{KS}]$,
\[ H^2(\ham(X),\R)^{\smp(X)}=\R [\psi_{KS}] \,.\]
\end{Remark}

In view of the different character of $H^2(\ham(X), \R)$ for compact and noncompact 
manifolds, we will treat these two cases separately.

\subsection{Transgression for compact manifolds}\label{sec:ncptsymp}

We first determine the continuous second Lie algebra cohomology of $\mathrm{sp}(X)$
in the case that $X$ is a \emph{compact} connected manifold, and return to 
noncompact manifolds in the next section.
In view of eqn.~\eqref{kert}, Cor.~\ref{spinv}, and Thm.~\ref{H2Ham}, 
this amounts to determining the
transgression map 
$T \colon H^1_{\mathrm{dR}}(X) \rightarrow {\textstyle \bigwedge^3}H^1_{\mathrm{dR}}(X)^*$.
This was done in \cite{Vizman2006}, using the characteristic class 
one obtains from the Weil homomorphism applied 
to an invariant bilinear form on $\ham(X)$.

We now give 
a direct construction of the transgression map, which  
avoids the use of characteristic classes, and therefore generalizes more easily to 
noncompact manifolds, where an invariant bilinear form is not available.
 
% \eqref{fgh}.

\begin{Proposition}\label{Tcompact}
For $X$ a compact connected symplectic manifold, the  transgression map
in the five term exact sequence \eqref{5term}
 is given by
\begin{gather}
T \colon H^1_{\mathrm{dR}}(X) \rightarrow {\textstyle \bigwedge^3}H^1_{\mathrm{dR}}(X)^*\nonumber
\\T(a)(b_1,b_2,b_3) =
(a,b_1,b_2,b_3)
- \frac{1}{\mathrm{vol}(X)} \Big(\sum_{\mathrm{cycl}}(a,b_1)(b_2, b_3)\Big) \label{final}
\end{gather}
with the alternating 4-form on $H^1_{\mathrm{dR}}(X)$ defined by
\begin{equation}\label{eq:alt4form}
([\al],[\beta_1],[\beta_2],[\beta_3])=\int_X \al \wedge\beta_1\wedge\beta_2\wedge\beta_3 \wedge\om^{n-2}/(n-2)!
\end{equation}
and the skew symmetric 2-form $(\,\cdot\,,\,\cdot\,)$ on $H^1_{\mathrm{dR}}(X)$ as in eqn.~\eqref{pair-nonII}.
\end{Proposition}
Note that one can always rescale $\omega$ so that $\mathrm{vol}(X)=1$, which was implicitly done in the derivation of
\cite[eqn.~(8)]{Vizman2006}.

\begin{proof}
By Corollary \ref{spinv}, we have 
\[H^2(\ham(X),\R)^{\mathrm{sp}(X)}=H^2(\ham(X),\R),\] and by 
Theorem \ref{H2Ham}, $H^2(\ham(X),\R)$ can be identified with $H^1_{\mathrm{dR}}(X)$ 
by mapping $[\alpha] \in H^1_{\mathrm{dR}}(X)$ to the class $[\psi_{\alpha}] \in H^2(\ham(X),\R)$
of the cocycle $\psi_{\alpha}$ described in eqn.~\eqref{voorham}.
The quotient $\mathrm{sp}(X)/\ham(X)=H^1_{\mathrm{dR}}(X)$ is an abelian Lie algebra, so
the transgression is a map 
$T\colon H^1_{\mathrm{dR}}(X)\to{\textstyle \bigwedge^3} H^1_{\mathrm{dR}}(X)^*$.

Since $X$ is compact, 
we can identify $\ham(X)$ with the Lie algebra $C^\oo_{c,0}(X)$ of zero integral
functions by mapping $X_f$ to its unique zero integral Hamilton function 
$f - \frac{1}{\mathrm{vol}(X)}\langle f \rangle$. 
Because the Hamiltonian function $\{f,g\}$ of $[X_f,X_g]$ has zero integral, we see from
eqn.~\eqref{grondeekhoorn} that
$L_{v}\psi_{\alpha} = \dd (\Theta_{\alpha}(v))$, with the 1-cochain
$\Theta_\alpha:\mathrm{sp}(X)\to\ham(X)'$
given by
$$
\Th_\al(v)(X_f)=\int_X\alpha(v)\left(f - {\textstyle \frac{1}{\mathrm{vol}(X)}}\langle f\rangle\right)\omega^{n}/n!,
$$
where $\langle f \rangle$ is given by \eqref{average} and $\mathrm{vol}(X) = \langle 1 \rangle$
is the symplectic volume of $X$.
To compute the transgression map, we proceed as outlined in Section~\ref{sec:5term}. 
We choose a 2-cochain $\psi'_\al$ 
on $\mathrm{sp}(X)$ that extends $\Theta_\al \colon \smp(X) \times \ham(X) \rightarrow \R$, and
compute its differential. 
Using Lemma~\ref{lem:4termcyc} and 
the fact that the Lie bracket of two symplectic vector fields
$v_1$ and $v_2$ is Hamiltonian with 
Hamilton function 
$\om(v_1,v_2)$,
we compute the differential of $\psi'_\al$ as
\begin{eqnarray*}
(\dd\psi'_\al)(v_1,v_2,v_3) &=& \sum_{cycl}\psi'_\al(v_1,[v_2,v_3])
= \sum_{\mathrm{cycl}}\Theta_\al(v_1)([v_2,v_3])\\
&=& \sum_{\mathrm{cycl}}\int_X\al(v_1)\left(\om(v_2,v_3)- 
{\textstyle \frac{1}{\mathrm{vol}(X)} } \langle \om(v_2,v_3)\rangle \right) \omega^n/n!\\
&=& \int_X
\alpha\wedge i_{v_1}\om\wedge  i_{v_2}\om\wedge  i_{v_3}\om
\wedge\om^{n-2}/(n-2)!\\
& & - \frac{1}{\mathrm{vol}(X)} \sum_{\mathrm{cycl}}
\langle \alpha(v_1) \rangle
\langle \omega(v_2,v_3)\rangle \\
%
%\int_X\om(v_2,v_3)\om^{n}
%\int_X\al(v_1)\om^{ n}\\
%
& = & 
([\al],[i_{v_1}\om], [i_{v_2}\om], [i_{v_3}\om])\\
& &-\frac{1}{\mathrm{vol}(X)} \sum_{\mathrm{cycl}} ([\al],[i_{v_1}\omega])
([i_{v_2}\om],[i_{v_3}\om])\,. \\
%& & +([\al],[i_{v_1}\om], [i_{v_2}\om], [i_{v_3}\om]).
\end{eqnarray*}
With the expression $T([\psi_\al]) := [\ol{\dd \psi'_\al}]$ of the transgression map, the conclusion follows.
\end{proof}

Knowing  this explicit expression for the
transgression map, equation \eqref{kert} yields 
the following complete description for 
$H^2(\mathrm{sp}(X),\R)$.

\begin{Theorem}\label{thm:spcpt}
Let $X$ be a compact connected symplectic manifold.
Then we have an isomorphism
\[
H^2(\mathrm{sp}(X),\RR) \simeq {\textstyle \bigwedge^2}(H^1_{\mathrm{dR}}(X)^*)
\oplus \mathrm{Ker}(T),
\]
where $\mathrm{Ker}(T) \subseteq H^1_{\mathrm{dR}}(X)$
is the set of classes $a\in H^1_{\mathrm{dR}}(X)$ such that
\[
(a,b_1,b_2,b_3) = 
\frac{1}{\mathrm{vol}(X)}
\sum_{\mathrm{cycl}}(a,b_1)(b_2,b_3)
\]
for all $b_1, b_2, b_3\in H^1_{\mathrm{dR}}(X)$.
\end{Theorem}

We give a few examples to illustrate that, in concrete situations, the kernel 
of the transgression can often be determined explicitly. 

\begin{Example}[Tori]
For the symplectic $2n$-torus $\mathbb{T}^{2n} = \R^{2n}/\Z^{2n}$, 
%with symplectic
%form $\omega = \sum_{i=1}^{n}dx_i \wedge dp_{i}$, 
the transgression map vanishes, because 
$\smp(\mathbb{T}^{2n}) \simeq \ham(\mathbb{T}^{2n}) \rtimes \R^{2n}$ 
is a semidirect product, cf.\ \cite{Vizman2006}. 
Therefore, all Roger cocycles extend to $\smp(\mathbb{T}^{2n})$, 
and $H^2(\smp(\mathbb{T}^{2n}),\R) \simeq \bigwedge^2(\R^{2n}) \oplus \R^{2n}$.
\end{Example}

\begin{Example}[Compact surfaces]
For a surface $\Sigma_{g}$ of genus $g \ge 2$, the transgression map is injective;  
although the 4-form \eqref{eq:alt4form} is zero, the cyclic sum involving the nondegenerate 2-form \eqref{pair-nonII} 
does not vanish, cf.\ \cite{Vizman2006}.
Therefore no Roger cocycle can be extended to $\smp(\Sigma_{g})$,  
and $H^2(\smp(\Sigma_{g}),\R) \simeq \bigwedge^2 \R^{2g}$. 
For genus $0$, we have $H^2(\smp(S^2),\R) = \{0\}$, 
and for genus $1$, we have $H^2(\smp(\mathbb{T}^2),\R) \simeq \bigwedge^2 (\R^2) \oplus \R^{2}$ by the previous remark.  
\end{Example}

\begin{Remark}[Compact manifolds with $b_1 < 4$] \label{Rk:smallbetti}
Since the transgression \eqref{final} %$T \colon H^1_{\mathrm{dR}}(X) \rightarrow \bigwedge^{3}H^1_{\mathrm{dR}}(X)^*$
%for a compact manifold $X$
induces an \emph{alternating} map $\bigwedge^{4}H^1_{\mathrm{dR}}(X) \rightarrow \R$, it vanishes
for dimensional reasons if the first Betti number satisfies $b_1 <4$. 
In this case, all Roger cocycles extend to $\smp(X)$, and $H^2(\smp(X),\R) \simeq \bigwedge^2(\R^{b_1}) \oplus \R^{b_1}$.
\end{Remark}

\begin{Example}\label{Thurston}
\emph{Thurston's manifold} $X^4$ %is a 4-dimensional, 
was historically the first example of a
compact, symplectic manifold
which is not K\"ahler \cite{Thurston1976}. 
It is a \emph{nilmanifold}, i.e.,  a quotient $X = \Gamma \backslash N$
of a 1-connected, nilpotent Lie group $N$ by a discrete, cocompact subgroup $\Gamma$.
For Thurston's manifold,
%In this case, 
the nilpotent Lie group is $N = H \times \R$, where $H$ is the \emph{Heisenberg group}.
Its Lie algebra is $\n = \fh \times \R$, where 
$\mathfrak{h} = \mathrm{Span}\left<x,p,h\right>$
is the \emph{Heisenberg Lie algebra} with 
$[x,p] = h$ and $[h,x] = [h,p] = 0$, and $\R$ is central in $\n$ with generator $z$.
We denote by $x^*, p^*, h^*, z^*$ the dual basis of $x, p, h, z$. 
Under the inclusion 
$\iota \colon \bigwedge \n^* \hookrightarrow \Omega(X^4)$ as left invariant forms,
the symplectic form is the image of
\[\omega = h^* \wedge x^* +   z^* \wedge p^*\,,\]
normalised so that the symplectic volume $\mathrm{vol}(X^4)$ is 1.
% $\omega^2/2! = h^*\wedge x^*\wedge z^*\wedge p^*$ integrates to 1.
By a theorem of Nomizu \cite{No54}, the inclusion $\iota \colon \bigwedge \n^* \hookrightarrow \Omega(X^4)$
yields an isomorphism 
between the Lie algebra cohomology $H^{\bullet}(\n,\R)$  and the de Rham cohomology 
$H^{\bullet}_{\rm dR}(X^4)$.
The Lie algebra differential $\dd \colon \bigwedge \n^* \rightarrow \bigwedge \n^*$ is the derivation with   
$\dd h^* = - x^* \wedge p^*$ and $\dd x^* = \dd p^* = \dd z^* = 0$. From this, one finds that
 the cohomology 
$H^{\bullet}(\n,\R)$ is generated (as a ring) by $[x^*]$, $[p^*]$ and $[z^*]$ in degree 1 
and $[h^* \wedge p^*]$, $[h^* \wedge x^*]$ in degree 2, with the single relation 
$[x^* \wedge p^*] = 0$.
The fact that
\begin{equation}\label{HoneThurston}
H_{\dR}^1(X^4) =  \mathrm{Span} \left<[x^*],[p^*], [z^*]\right>
\end{equation}
is 3-dimensional led Thurston to his conclusion that $X^4$ does not admit a K\"ahler structure.
By Remark \ref{Rk:smallbetti}, it also implies that every Roger cocycle extends to $\smp(X^4)$, and 
$H^2(\smp(X^4),\R) \simeq \bigwedge^2 (\R^3)  \oplus \R^3$.
\end{Example}

\subsection{Transgression for noncompact manifolds}
We now determine the continuous second Lie algebra cohomology of $\mathrm{sp}(X)$
in the case that $X$ is a \emph{noncompact} connected manifold.
As in the previous section, this amounts to
determining the
transgression map 
\[T \colon H^2(\ham(X),\R)^{\mathrm{sp}(X)} \rightarrow 
{\textstyle \bigwedge^3}H^1_{\mathrm{dR}}(X)^*,\]
but in contrast to the compact case, not every class 
$[\psi_\alpha] \in H^2(\ham(X), \R)$
is $\smp(X)$-invariant.
By Cor.~\ref{spinv-non}, $[\psi_\alpha]$
is $\mathrm{sp}(X)$-invariant if and only if $[\al]\in H^1_{\mathrm{dR}, c}(X)$ 
lies in the kernel
of the alternating pairing \eqref{pair-nonII}, and an 
$\mathrm{sp}(X)$-invariant class $[\psi_\alpha]$
extends from $\ham(X)$ to $\mathrm{sp}(X)$ if and only if it
lies in the kernel of the transgression map.

\begin{Proposition}\label{Tnoncompact}
The transgression map 
\[
T\colon H^2(\ham(X),\R)^{\mathrm{sp}(X)}\to\Lambda^3 H^1_{\mathrm{dR}}(X)^*
\]
is given by
\be \label{finalnon}
T([\psi_\al])(b_1,b_2,b_3)=([\al],b_1,b_2,b_3)
\ee
for the 4-linear map on 
$H^1_{\mathrm{dR,c}}(X) \times H^1_{\mathrm{dR}}(X)^3  $ defined by
\be\label{4linmap}
(a,b_1,b_2,b_3)= \int_X  a \wedge b_1\wedge b_2\wedge b_3
\wedge\om^{n-2}/(n-2)!\,.
\ee
\end{Proposition}

\begin{proof}
We proceed along the lines of Prop.~\ref{Tcompact}.
By Corollary \ref{spinv-non}, the class $[\psi_{\alpha}] \in H^2(\ham(X),\R)$
is annihilated by $\smp(X)$ if and only if $[\al] \in H^1_{\mathrm{dR},c}(X)$
%Since $[\al]$ 
lies in the kernel
of the alternating pairing \eqref{pair-nonII}. 
As $\langle \alpha(v)\rangle$ can be expressed in terms of the pairing as 
$([\alpha],[i_{v}\omega])$, this is equivalent to the requirement that   
$\langle \alpha(v)\rangle = 0$ for all $v\in \smp(X)$.
%We therefore have $\langle \alpha(v)\rangle = 0$ for all $v\in \smp(X)$ so that
By eqn.~\eqref{worteltaart}, we thus have
%from the proof of Proposition \ref{slartibartfast} that
$L_v\psi_\al= \dd(\th_\al(v))$ for all $v\in\mathrm{sp}(X)$,
where the 1-cochain $\th_\al:\mathrm{sp}(X)\to\ham(X)'$ is defined by  \eqref{teta-non}.
The differential of any 2-cochain 
$\psi'_\al \colon \mathrm{sp}(X) \times \mathrm{sp}(X)\rightarrow \R$ 
that extends $\th_\al \colon \smp(X) \times \ham(X)\rightarrow \R$ is
\begin{align*}
(\dd\psi'_\al)(v_1,v_2,v_3)&=\sum_{\mathrm{cycl}}\psi'_\al(v_1,[v_2,v_3])
=\sum_{\mathrm{cycl}}\theta_\al(v_1)([v_2,v_3])\\
&=\sum_{\mathrm{cycl}}\int_X\al(v_1)\left(\om(v_2,v_3)-
\om(v_2,v_3)_x\right)\om^{n}/n!\\
&=\sum_{\mathrm{cycl}}\int_X\al(v_1)\om(v_2,v_3)\om^{n}/n! \\
%- \sum_{\mathrm{cycl}}\langle \alpha(v_1) \rangle\om(v_2,v_3)_x\\
%\int_X\al\wedge i_{v_1}\om\wedge\om^{ (n-1)}\\
&\stackrel{\eqref{forma}}{=}
\int_X
\alpha \wedge i_{v_1}\om\wedge  i_{v_2}\om\wedge  i_{v_3}\om
\wedge\om^{n-2}/(n-2)!.
\end{align*}
In the third step, 
we use that the commutator $[v_1,v_2]$ of symplectic vector fields 
is Hamiltonian 
with Hamiltonian function 
$\om(v_1,v_2)$, while in the fourth step, we use that $\langle\al(v)\rangle=\int_{X}\alpha(v)\omega^{n}/n!$
is zero by $\smp(X)$-invariance of $[\psi_{\alpha}]$.
%, and in the last step, we used Lemma~\ref{lem:4termcyc}.
The required expression of the transgression map follows.
\end{proof}

If we consider the bilinear map 
$(\,\cdot\,,\,\cdot\,)$ of eqn.~\eqref{pair-nonII} as a linear map
\[B \colon H^1_{\mathrm{dR},c}(X) 
\rightarrow H^1_{\mathrm{dR}}(X)^*\,,\] 
and the 4-linear map 
$(\,\cdot\,,\,\cdot\,,\,\cdot\,,\,\cdot\,)$ as minus the transgression map
\[T \colon H^1_{\mathrm{dR},c}(X) 
\rightarrow {\textstyle \bigwedge^3}H^1_{\mathrm{dR}}(X)^*,\] 
then $[\psi_{\alpha}]$ extends to $\smp(X)$ if and only if 
$[\alpha] \in H^1_{\mathrm{dR},c}(X)$ lies in
$\mathrm{Ker}(B)\cap\mathrm{Ker}(T)$.
Since the Kostant-Souriau class $[\psi_{KS}]$
can be described as $[\psi_{\alpha}]$ with $\alpha = dh$ and
$\langle h \rangle = 1$, it always extends to a class  on $\smp(X)$.
Indeed, an extension is given by the cocycle
\[
\psi'_{KS}(v,w)=\omega(v,w)_x.
\] 
This yields the desired explicit 
description of the second continuous Lie algebra cohomology
of $\smp(X)$. 

\begin{Theorem}\label{thm:spncpt}
Let $X$ be a noncompact connected symplectic manifold. Then we have an isomorphism
\[H^2(\smp(X),\R) \simeq {\textstyle \bigwedge^{2}}(H^1_{\mathrm{dR}}(X)^*)
\oplus \R [\psi'_{KS}]
\oplus (\mathrm{Ker}(B) \cap \mathrm{Ker}(T))\,,
\]
where 
%$\mathrm{Ker}(T:\mathrm{Ker}B\to\La^3H^1_{\dR}(X)^*)
$\mathrm{Ker}(B) \cap \mathrm{Ker}(T)
\subseteq H^1_{\mathrm{dR},c}(X)$ 
is the set of classes $a$ such that
\begin{eqnarray*}
(a,b) &=& 0\,,\\
(a,b_1,b_2,b_3) &=& 0
\end{eqnarray*}
for all $b$ and $b_1,b_2,b_3$ in $H^1_{\mathrm{dR}}(X)$.
\end{Theorem}
We apply Theorem \ref{thm:spncpt} to a few
(classes of) examples, where we calculate
 $\mathrm{Ker}(B)$ and $\mathrm{Ker}(T)$
 using standard methods 
in algebraic topology.

\begin{Example}[Noncompact surfaces]\label{ncptsurface}
For a 2-dimensional surface $\Sigma$, the map 
$B \colon H^1_{\mathrm{dR},c}(\Sigma) \rightarrow H^1_{\mathrm{dR}}(\Sigma)^*$
is an isomorphism by Poincar\'e duality. If $\Sigma$ is noncompact, 
Theorem~\ref{thm:spncpt} yields
$H^2(\smp(\Sigma),\R) \simeq {\textstyle \bigwedge^{2}}(H^1_{\mathrm{dR}}(\Sigma)^*)
\oplus \R [\psi'_{KS}]$ .
\end{Example}
Also for cotangent spaces, the Roger cocycles are trivial in cohomology. 
\begin{Corollary}
Let $T^*Q$ be the cotangent space of a connected manifold $Q$.
Then all 
Roger cocycles are trivial, and there is 
an isomorphism 
\[H^2(\smp(T^*Q),\R) \simeq {\textstyle \bigwedge^2 }(H^1_{\mathrm{dR}}(Q)^*) \oplus \R[\psi'_{KS}]\,.\]
\end{Corollary}

\begin{proof}
By Poincar\'e duality, 
$H^{1}_{\mathrm{dR},c}(T^*Q)$ is isomorphic to 
$H^{2n-1}_{\mathrm{dR}}(T^*Q)^*$.
Since $Q$ is a deformation retract of $T^*Q$, 
we have
$H^{2n-1}_{\mathrm{dR}}(T^*Q) \simeq H^{2n-1}_{\mathrm{dR}}(Q)$,
so $H^{1}_{\mathrm{dR},c}(T^*Q) \simeq H^{2n-1}_{\mathrm{dR}}(T^*Q)^* \simeq H^{2n-1}_{\mathrm{dR}}(Q)^*$ 
vanishes for $n>1$.
For $n=1$, $H^{1}_{\mathrm{dR},c}(T^*Q)$ need not vanish, but 
$\mathrm{Ker}(B) = \{0\}$ by Example \ref{ncptsurface}.
Since $H^{1}_{\mathrm{dR}}(Q)^* \simeq H^{1}_{\mathrm{dR}}(T^*Q)^*$,
this concludes the proof.  
\end{proof}

\begin{Remark}[Punctured symplectic manifolds]\label{Rk:punctures}
Suppose that $X = M - \{m\}$ is obtained from a \emph{compact} symplectic manifold $M$
by removing a point $m\in M$. 
Then the Mayer-Vietoris sequence for de Rham cohomology shows that 
the pullback $\iota^* \colon H^1_{\mathrm{dR}}(M) \rightarrow H_{\mathrm{dR}}^1(X)$ 
by the inclusion 
$\iota \colon X \hookrightarrow M$ is an isomorphism. Similarly, the pushforward
$\iota_{*} \colon H^1_{\mathrm{dR},c}(X) \rightarrow H^1_{\mathrm{dR}}(M)$ 
is an isomorphism by
the Mayer-Vietoris sequence 
for compactly supported de Rham cohomology.
Since every class $[\alpha] \in H^1_{\mathrm{dR}}(X)$ can be represented by a compactly
supported 1-form $\alpha$, the maps 
$H^1_{\mathrm{dR}, c}(X) \times H^1_{\mathrm{dR}}(X) \rightarrow \R$
and
$H^1_{\mathrm{dR}, c}(X) \times \bigwedge^{3}H^1_{\mathrm{dR}}(X) \rightarrow \R$
induced by equation 
\eqref{pair-nonII}
and
\eqref{4linmap}
agree with the corresponding alternating forms $\bigwedge^{2}H_{\mathrm{dR}}^1(M) \rightarrow \R$
and $\bigwedge^{4}H^1_{\mathrm{dR}}(M) \rightarrow \R$ for the compact manifold $M$.
\end{Remark}

The following situation affords an example where some, 
but not all, Roger cocycles extend to the Lie algebra of symplectic vector fields.

\begin{Example}[Thurston's manifold with puncture]
Consider the noncompact symplectic manifold
$X^4_{m} := X^{4} - \{m\}$, where  
$X^4$ is Thurston's symplectic manifold of Example \ref{Thurston}.
By Remark \ref{Rk:punctures}, we can determine $\mathrm{Ker}(B)$
and $\mathrm{Ker}(T)$ from the cohomology ring $H^{\bullet}_{\mathrm{dR}}(X^4)$
of the compact manifold,
which was calculated in Example \ref{Thurston}.
Recall from equation \eqref{HoneThurston}
that
\[
H_{\dR}^1(X^4) =  \left <[x^*],[p^*], [z^*]\right>\,.
\]
Using the formula
$\omega=  h^* \wedge x^* +   z^* \wedge p^*$
and the fact that Liouville form $\omega^2/2 =   z^* \wedge p^* \wedge h^* \wedge x^*$ integrates to 1,  
we find
\begin{equation}\label{xpz}
([x^*],[p^*])=([x^*],[z^*])=0,\quad ([z^*],[p^*])= 1
\end{equation}
for the alternating pairing \eqref{pair-nonII}.
It follows that $\mathrm{Ker}(B) = \R[x^*]$ is 1-dimensional,
and the $\smp(X^4_{m})$-invariant part of $H^2(\ham(X^4_{m}),\R)$
is spanned by a single class $[\psi_{\widetilde{x}^*}]$,
with $\widetilde{x}^*$ a compactly supported 1-form on $X^4_{m}$ cohomologous to $x^*$.
%where $\widetilde{x}^*$ is a compactly supported 1-form on $X^4-\{p\}$ cohomologous to $x^*$.
Since the alternating 4-linear map \eqref{eq:alt4form} on $H_{\mathrm{dR}}^1(X^4)$
vanishes for dimensional reasons, the invariant class $[\psi_{\widetilde{x}^*}]
\in H^2(\ham(X^4_{m}),\R)$
extends 
to $[\psi'_{\widetilde{x}^*}] \in H^2(\smp(X^4_{m}),\R)$. For the punctured Thurston manifold $X^{4}_{m}$, we thus find
\[H^2(\smp(X^4_{m}),\R) \simeq {\textstyle \bigwedge^2(\R^3)} \oplus \R[\psi'_{KS}] 
\oplus \R [\psi'_{\widetilde{x}^*}]\,.\]
\end{Example}

\section*{Acknowledgements}
C.V.\ was supported by the grant PN-II-ID-PCE-2011-3-0921
of the Romanian National Authority for Scientific Research.
B.J.\ was supported by the NWO grant 
613.001.214 ``Generalised Lie algebra sheaves''.

This is a pre-copyedited, author-produced version of an article accepted for publication in International Mathematical Research Notices following peer review. The version of record, Int.\ Math.\ Res.\ Not. (2016), Vol. 2016, No.\ 16, pp.\ 4996--5047
is available online at: 
{\tt https://doi.org/10.1093/imrn/rnv301}.

\bibliographystyle{alpha}
\bibliography{lijst}

\end{document}